\documentclass[a4paper,10pt,reqno]{amsart}
\usepackage{amssymb,amsmath,mathrsfs,graphics,graphicx,mathptm,float,lscape}
\usepackage[T1]{fontenc}
\usepackage[english, francais]{babel}
\usepackage[all]{xy}
\usepackage{lscape}
\usepackage{here}
\usepackage[utf8]{inputenc}
\usepackage{amssymb}
\usepackage{pgf,tikz}
\usetikzlibrary{arrows}

\newtheoremstyle{mytheoremstyle} % name
    {10 pt}                   % Space above
    {10 pt}                    % Space below
    {\slshape}                   % Body font
    {}                           % Indent amount
    {\bfseries}                   % Theorem head font
    {.}                          % Punctuation after theorem head
    {.5em}                       % Space after theorem head
    {}  % Theorem head spec (can be left empty, meaning ‘normal’)

\theoremstyle{mytheoremstyle}
\newtheorem{theo}{Theorem}[section]
\newtheorem{prop}[theo]{Proposition}
\newtheorem{lem}[theo]{Lemma}

\newtheorem{cor}[theo]{Corollary}
\newtheorem{theoalph}{Theorem}

\newtheorem{proalph}[theoalph]{Proposition}

\theoremstyle{definition}

\newtheorem{rem}[theo]{Remark}

\usepackage[colorlinks, linktocpage, citecolor = blue, linkcolor = blue]{hyperref}

\def\C{\mathbb{C}}
\def\a{\alpha}
\def\Z{\mathbb{Z}}
\def\N{\mathbb{N}}

\def\restriction#1#2{\mathchoice
              {\setbox1\hbox{${\displaystyle #1}_{\scriptstyle #2}$}
              \restrictionaux{#1}{#2}}
              {\setbox1\hbox{${\textstyle #1}_{\scriptstyle #2}$}
              \restrictionaux{#1}{#2}}
              {\setbox1\hbox{${\scriptstyle #1}_{\scriptscriptstyle #2}$}
              \restrictionaux{#1}{#2}}
              {\setbox1\hbox{${\scriptscriptstyle #1}_{\scriptscriptstyle #2}$}
              \restrictionaux{#1}{#2}}}
\def\restrictionaux#1#2{{#1\,\smash{\vrule height .8\ht1 depth .85\dp1}}_{\,#2}}

\addtocounter{section}{0}             % Start with section 1
\numberwithin{equation}{section}

\begin{document}

\selectlanguage{english}

\title[Dynamics of a family of polynomial automorphisms of $\mathbb{C}^3$, a phase transition]{Dynamics of a family of polynomial automorphisms of $\mathbb{C}^3$, \\ a phase transition}

\author{Julie D\'eserti}

\author{Martin Leguil}

\address{Universit\'e Paris Diderot, Sorbonne Paris Cit\'e, Institut de Math\'ematiques de
Jussieu-Paris Rive Gauche, UMR $7586$, CNRS, Sorbonne Universit\'es, UPMC Univ Paris $06$,
F-$75013$ Paris, France.}

\email{deserti@math.univ-paris-diderot.fr}

\email{martin.leguil@imj-prg.fr}

\begin{abstract}
The polynomial automorphisms of the affine plane have been studied a lot: 
if $f$ is such an automorphism, then either $f$
preserves a rational fibration, has an uncountable centralizer and its first dynamical
degree equals $1$, or $f$ preserves no rational curves, has a countable centralizer
and its first dynamical degree is $>1$. In higher dimensions there is no such description.
In this article we study a family $(\Psi_\alpha)_\alpha$ of polynomial automorphisms of 
$\mathbb{C}^3$. We show that the first dynamical degree of $\Psi_\alpha$ is $>1$, that 
$\Psi_\alpha$ preserves a unique rational fibration and has an uncountable centralizer. 
We then describe the dynamics of the family $(\Psi_\alpha)_\alpha$, in particular the 
speed of points escaping to infinity. We also observe different behaviors according 
to the value of the parameter $\alpha$.
\end{abstract}

\maketitle

\smallskip

\tableofcontents

\section{Introduction}\label{sec:intro}

H\'enon gave families of quadratic automorphisms of the plane 
which provide examples of dynamical systems with very complicated dynamics
(\cite{Henon1, Henon2, BenedicksCarleson, DevaneyNitecki}). In 
\cite{FriedlandMilnor} Friedland and Milnor proved that if $f$ belongs 
to the group $\mathrm{Aut}(\mathbb{C}^2)$ of polynomial automorphisms 
of $\mathbb{C}^2$, then $f$ is $\mathrm{Aut}(\mathbb{C}^2)$-conjugate 
either to an elementary automorphism (elementary in the sense that they
preserve a rational fibration), or to an automorphism of H\'enon 
type, \emph{i.e.\ }to
\[
g_1g_2\ldots g_k,\,\qquad\,g_j\colon(z_0,z_1)\mapsto(z_1,P_j(z_1)-\delta_j 
z_0),\,\delta_j\in\mathbb{C}^*,\, P_j\in\mathbb{C}[z_1],\,\deg P_j\geq 2.
\]
The topological entropy allows to measure chaotic behaviors. In 
dimension $1$ the topological entropy of a rational fraction
coincides with the logarithm of its degree; but the algebraic
degree of a polynomial automorphism of $\mathbb{C}^2$ is not 
invariant under conjugacy so \cite{RussakovskiiShiffman, Friedland} 
introduce the first dynamical degree. The topological entropy
is equal to the logarithm of the first dynamical degree 
(\cite{Smillie, BedfordSmillie}). If $f$ is of H\'enon type, then
the first dynamical degree of $f$ is equal to its algebraic 
degree $\geq 2$; on the other hand if $f$ is conjugate to an 
elementary automorphism, then its first dynamical degree is $1$
(\emph{see} \cite{FriedlandMilnor}). Another criterion to 
measure chaos is the size of the centralizer of an element. 
The group $\mathrm{Aut}(\mathbb{C}^2)$ has a structure of 
amalgamated product (\cite{Jung}) hence according to
\cite{BassSerre} this group acts non trivially on a tree. 
Using this action Lamy proved that a polynomial automorphism is
of H\'enon type if and only if its centralizer is countable 
(\cite{Lamy}). 

The group $\mathrm{Aut}(\mathbb{C}^3)$ and the dynamics of its 
elements are much less-known. In this article we study the 
properties of the family of polynomial automorphisms of 
$\mathbb{C}^3$ given by 
\[
\Psi_\a\colon (z_0,z_1,z_2) \mapsto (z_0+z_1+z_0^q z_2^d,z_0,\a z_2)
\]
where $\alpha$ denotes a nonzero complex number with modulus $\leq 1$, 
$q$ an integer $\geq 2$, and $d$ an integer $\geq 1$.

The automorphism $\Psi_\a$ can be seen as 
a skew product over the map $z_2 \mapsto \a z_2$, and whose dynamics 
in the fibers is given by automorphisms of Hénon type. More precisely, if $z_2 \in \C$, 
let us denote $\psi_{z_2}\colon(z_0,z_1) \mapsto (z_0+z_1+z_0^q z_2^d,z_0)$; 
then $\Psi_\a(z_0,z_1,z_2)=(\psi_{z_2}(z_0,z_1),\alpha z_2)$, and for every $n \geq 1$, we have
$\Psi_\a^n(z_0,z_1,z_2)=((\psi_{z_2})_n(z_0,z_1),\alpha^n z_2)$, where 
\begin{equation}
(\psi_{z_2})_n=\psi_{\a^{n-1}z_2} \circ \dots \circ \psi_{\a z_2} \circ \psi_{z_2}.
\end{equation}
If $\alpha \neq 0$, we also define the map $\phi_\a:=\alpha^l \psi_1$ where $l:=d/(q-1)$. We will see later on that $\Psi_\a$ is semi-conjugate to $\phi_\a$. 
The family of automorphisms $\{\Psi_\alpha\}_\a$ satisfies the following properties:

\begin{proalph}
Take $0< |\alpha|\leq 1$. Then 
\begin{itemize}
\smallskip
\item the first dynamical degree of the automorphism $\Psi_\alpha$ $($resp. $\Psi_\alpha^{-1})$ is $q\geq 2$;
\smallskip
\item the centralizer of $\Psi_\alpha$ is uncountable;
\smallskip
\item if $0<|\alpha|\leq 1$, then $\Psi_\alpha$ preserves a unique rational fibration, $\{z_2=\text{cst}\}$.
\end{itemize}
\end{proalph}

We then focus on the dynamics of $\Psi_\a$, $0< |\alpha|\leq 1$. Let us introduce 
a definition. We denote by $\varphi:=\frac{1+\sqrt{5}}{2}$ the golden ratio. We say that the forward orbit of $p$ \textit{goes
 to infinity with Fibonacci speed} if the sequence 
 $(\Psi_\a^n(p)\varphi^{-n})_{n \geq 0}$ converges and 
 $\lim\limits_{n \to +\infty} \Psi_\a^n(p)\varphi^{-n} = p' \neq 0_{\C^3}$. 
 In particular this implies
 \[
 \|\Psi_\a^n(p)\| \sim \|p'\| \varphi^n.
 \]
The hypersurface $\{z_2=0\}$ is fixed by $\Psi_\a$, 
and the induced map on it is a linear Anosov diffeomorphism. We see that for any 
$p \in \{z_2=0\}$, either its forward orbit goes to $0_{\C^3}$ exponentially fast, or 
it escapes to infinity with Fibonacci speed.
Concerning points escaping to infinity with maximal speed, we prove:
\begin{theoalph}
Fix $0<|\alpha|\leq 1$. For any point $p \in \C^3$, the limit 
$\lim\limits_{n\to +\infty}\frac{\log^+\vert\vert\Psi_\alpha^n(p)\vert\vert}{q^n}$
exists. The function 
\[
G^+_{\Psi_\alpha}(p)=\lim_{n\to +\infty}\frac{\log^+\vert\vert\Psi_\alpha^n(p)\vert\vert}{q^n}
\]
is plurisubharmonic, Hölder continuous, and satisfies 
$G_{\Psi_\a}^+\circ \Psi_\a = q \cdot G_{\Psi_\a}^+$. Set 
$\widetilde l := 2 \max\left(\frac{d}{q-1},1\right)$; then 
\[
1\leq \limsup\limits_{\|p\| \to +\infty} \frac{G_{\Psi_\a}^+(p)}{\log \|p\|}\leq \widetilde l.
\]

Moreover, the set $\mathcal{E}:=\big\{p \in \C^3\ \vert \ G_{\Psi_\a}^+(p) >0\big\}$ 
of points escaping to infinity with maximal speed is open, connected, 
and has infinite Lebesgue measure on any complex line where $G_{\Psi_\a}^+$ 
is not identically zero. 
In particular, the set 
$\big\{p \in \C^3\ \vert \ \lim\limits_{n \to +\infty}\|\Psi_\a^n(p)\| = +\infty\big\}$ is of 
infinite measure. 
We also exhibit an explicit open set $\Omega \subset \mathcal{E}$. 
\end{theoalph}

\begin{theoalph}
Assume $0 < |\alpha| < 1$. Then $\Psi_\a$ has a 
unique periodic point at finite distance, $0_{\C^3}=(0,0,0)$, which is a saddle
point of index $2$. The fixed hypersurface $\{z_2=0\}$ attracts any other point. 
Moreover, the set $K_{\Psi_\a}^+$ of points with bounded forward orbit is exactly 
the stable manifold $W_{\Psi_\a}^s(0_{\C^3})$, and the latter can be characterized
analytically. The set $J_{\Psi_\a}^+:=\partial K_{\Psi_\a}^+$ thus 
corresponds to $\overline{W_{\Psi_\a}^s(0_{\C^3})}$. 
\end{theoalph}

We observe a phase transition in the dynamics of the family $\{\Psi_\a\}_{0<|\a|\leq 1}$ 
for the value $|\a|= \varphi^{(1-q)/d}$:

\begin{theoalph}
Assume $0<\vert\alpha\vert<\varphi^{(1-q)/d}$. The set $\mathcal{V}:=\big\{p\in\C^3\ \vert\ G_{\Psi_\a}^+(p)=0\big\}$ 
is a closed neighborhood of the hyperplane $\{z_2=0\}$. It consists in the disjoint 
union $\Omega'' \sqcup W_{\Psi_\a}^s(0_{\C^3})$, where $\Omega''$ has non-empty interior 
and the forward orbit of any point $p \in \Omega''$ goes to infinity with Fibonacci speed. 
\end{theoalph}

Note that we also define an analytic function $g$ whose domain of definition is equal 
to $\mathcal{V}$, and which parametrizes the stable manifold in the sense that 
$W_{\Psi_\a}^s(0_{\C^3})$ coincides with the zero set $\mathcal{Z}$ of $g$.

\begin{theoalph}
Assume now $\varphi^{(1-q)/d}<\vert\alpha\vert<1$. For any $p \in \C^3$, 
exactly one of the following cases occurs:
\smallskip
\begin{itemize}
\item either $p \in W_{\Psi_\a}^s(0_{\C^3})$ and its forward orbit converges 
to $0_{\C^3}$ exponentially fast;
\smallskip
\item or $p \in \{z_2=0\} \smallsetminus W_{\Psi_\a}^s(0_{\C^3})$ and it goes to infinity with Fibonacci 
speed;
\smallskip
\item or the speed explodes: $G_{\Psi_\a}^+(p)>0$. 
\end{itemize}
\end{theoalph}

In particular, contrary to the previous situation where the set $\Omega''$ has non-empty interior, 
we see here that Fibonacci speed does not occur 
outside the hypersurface $\{z_2 = 0\}$.

\begin{rem}
We stress the fact that for any $0<|\a|<1$, the forward orbit of a point under $\Psi_\a$
 is bounded if 
and only if it goes to $0_{\C^3}$. Moreover, we see that in the case the orbit 
is unbounded, it has to escape to infinity. 
This rigidity phenomenon is related to the properties of the automorphism of Hénon 
type $\phi_\a$ to which 
$\Psi_\a$ is semi-conjugate, and which possesses an attractor at infinity such that 
the positive iterates of any point whose forward orbit is not bounded escape to it. 
\end{rem}

\begin{theoalph}
Assume $|\a|=1$.
We define $K_{\Psi_\a}$ to be the set of points $p \in \C^3$ whose orbit $(\Psi_\a^n(p))_{n \in \mathbb{Z}}$ is bounded. Similarly to what we did above, 
we define the Green function 
$G_{\Psi_\a}^-$. Then for any point $p \in \C^3$, exactly one of the following 
 assertions is satisfied:
\begin{itemize}
\smallskip
\item either the orbit of $p$ is bounded, i.e.\ $p \in K_{\Psi_\a}$;
\smallskip
\item or $p \in \{z_2=0\} \smallsetminus \{0_{\C^3}\}$ and either 
its forward orbit or its backward orbit escapes to infinity with Fibonacci speed;
\smallskip
\item or $G_{\Psi_\a}^+(p) >0$ or $G_{\Psi_\a}^-(p)>0$; in this case, 
either its forward orbit or its backward orbit escapes to infinity with maximal speed.
\end{itemize}
\smallskip

We define the associate Green currents
$T_{\Psi_\a}^\pm:=\mathrm{dd^c}G_{\Psi_\a}^\pm$, and 
we set $\mu_{\Psi_\a}:=T_{\Psi_\a}^+\wedge T_{\Psi_\a}^- \wedge dz_2 
\wedge d\overline{z_2}$. The measure $\mu_{\Psi_\a}$ is invariant by 
$\Psi_\a$ and supported on the Julia set $J_{\Psi_\a}:=\partial K_{\Psi_\a}$. 
For any $p_2 \neq 0$, the set 
$\mathcal{C}_{p_2}:=\mathbb{C}^2 \times \big\{p_2 e^{ix}\ \vert\ x \in \mathbb{R}\big\}$ 
is invariant under $\Psi_\a$. Define $\mathcal{J}_{p_2}:=J_{\Psi_\a} \cap \mathcal{C}_{p_2}$; 
it is also invariant and we show that when $\a$ is not a root of unity, 
$(\restriction{\Psi_\a}{\mathcal{J}_{p_2}}, \mu_{\Psi_\a})$ is ergodic.
\end{theoalph}

\subsection*{Acknowledgement} The first author would like to thank Artur Avila 
for helpful and fruitful discussions.

\section{Invariant fibrations and degree growths}

\subsection{Invariant fibrations}

Let us come back to dimension $2$ for a while. As recalled in \S \ref{sec:intro}
if $f$ is a polynomial automorphism of~$\mathbb{C}^2$, then up to conjugacy
either $f$ is an elementary automorphism or $f$ is of H\'enon type. In the 
first case $f$ preserves a rational fibration, whereas in the
second one $f$ does not preserve any rational curve~(\cite{Brunella}). 
This gives a geometric criterion to distinguish maps of H\'enon type and
elementary maps. What about dimension $n\geq 3$? Contrary to the 
$2$-dimensional case we will see, as soon as $n=3$, that "no invariant 
rational curve" does not mean "first dynamical degree $>1$".

Assume that $0<|\alpha|\leq 1$. 
Let $\Phi_\alpha$ be the automorphism of $\C^3$ given by 
\[
\Phi_\alpha=(\alpha^l (z_0 + z_1 + z_0^q),\alpha^l z_0,\alpha z_2)
\]
where $l:=\frac{d}{q-1}$. It is possible to show that $\Psi_\a$ is 
conjugate to $\Phi_\a$ through the \emph{birational map} of $\mathbb{P}^3_\mathbb{C}$
given in the affine chart $z_3=1$ by
\[ 
\theta=(z_0z_2^l,z_1 z_2^l,z_2), 
\]
that is $\theta \circ \Psi_\a = \Phi_\a \circ \theta.$
The advantage is that the action of $\Phi_\a$ in the fibers is independent of the 
base point. Moreover, it has a lot of good properties; in particular, we will see that 
it is algebraically stable (\S \ref{sec:degree}). Nevertheless $\theta$ is birational  
so we might loose some information (\S \ref{sec:lost}). 

\begin{prop}\label{pro:pasdefib}
For any $0<|\alpha|\leq 1$, the polynomial automorphism $\Phi_\alpha$ preserves a unique rational fibration, the 
fibration given by $\{z_2=\text{cst}\}$.
\end{prop}

\begin{cor}
For any $0<|\alpha|\leq 1$, the polynomial automorphism $\Psi_\alpha$ preserves a unique rational fibration, the 
fibration given by $\{z_2=\text{cst}\}$.
\end{cor}

\begin{proof}[Proof of Proposition \ref{pro:pasdefib}]
Note that $\Phi_\alpha=(\phi_\alpha(z_0,z_1),\alpha z_2)$ where 
$\phi_\alpha\in\mathrm{Aut}(\mathbb{C}^2)$ is the automorphism of H\'enon type
given by $\phi_\alpha\colon(z_0,z_1)\mapsto\alpha^l (z_1 + z_0 + z_0^q, z_0)$. Since 
$\phi_\a$ does not preserve rational curves (\cite{Brunella}) the 
only invariant rational fibration is $\{z_2=\text{cst}\}$.
\end{proof}

\subsection{Degrees and degree growths}\label{sec:degree}

The results in this part hold for any $\alpha \neq 0$.
As we saw in \S\ref{sec:intro} the first dynamical degree is an 
important invariant; in this section we will thus compute 
$\lambda(\Psi_\alpha^{\pm 1})$ and $\lambda(\Phi_\alpha^{\pm 1})$. 
Let us first mention 
a big difference between dimension $2$ and higher dimensions: 
if~$f$ belongs to $\mathrm{Aut}(\mathbb{C}^2)$ then $\deg f=\deg f^{-1}$.
This equality does not necessarily hold in higher dimension; 
nevertheless if $f$ belongs to $\mathrm{Aut}(\mathbb{C}^n)$, then
$\deg f\leq (\deg f^{-1})^{n-1}$ and $\deg f^{-1}\leq (\deg f)^{n-1}$
(\emph{see} \cite{Sibony}).

\begin{lem}\label{degpsi}
We have for any $n\geq 0$ both
\[
\mathrm{deg}(\Psi_\a^n)=q^n + d \times \frac{q^n-1}{q-1}
\]
and 
\[
\deg(\Psi_\alpha^{-n})=\deg(\Psi_\alpha^{n}).
\]
\end{lem}

\begin{proof}
Let us denote by $(P_\a^{(n)})_{n \geq -1}$ the sequence of polynomials 
where 
\[
\left\{
\begin{array}{lll}
P_\a^{(-1)}(z_0,z_1,z_2):=z_1 \\
P_\a^{(0)}(z_0,z_1,z_2):=z_0 \\ 
\forall\, n\geq 0\quad
P_\a^{(n+1)}:=P_\a^{(n)} + P_\a^{(n-1)} +(P_\a^{(n)})^q (\alpha^{n}z_2)^d.
\end{array}
\right.
\]
In particular, for every $n \geq 0$, 
$\Psi_\a^n(z_0,z_1,z_2) =\big(P_\a^{(n)}(z_0,z_1,z_2),P_\a^{(n-1)}(z_0,z_1,z_2),\a^n z_2\big)$. 
Since the degree of the third component does not change, and the 
second component is just the first one at time $n-1$, the growth 
of the degree is supported by the first component, that is 
$\mathrm{deg}(\Psi_\a^n)=\mathrm{deg}(P_\a^{(n)})$. Let us then show 
the result by induction on $n$. The result is true for $n =0$. 
If it holds for $n \geq 0$, then we have 
\[
\deg(P_\a^{(n+1)})=\deg\Big((P_\a^{(n)})^q (\alpha^{n}z_2)^d\Big)
=q\left(q^n + d \times \frac{q^n-1}{q-1}\right)+d=q^{n+1} + d 
\times \frac{q^{n+1}-1}{q-1}.
\]
\end{proof}

Since the degree is not invariant under conjugacy, 
\cite{RussakovskiiShiffman, Friedland} introduce the first dynamical 
degree. If $f$ is a polynomial automorphism of $\mathbb{C}^3$, 
the first dynamical degree of $f$ is defined by
\[
\lambda(f)=\displaystyle\lim_{n\to +\infty}(\deg f^n)^{1/n}.
\]
It satisfies the following inequalities: $1\leq\lambda(f)\leq\deg f$.

\begin{cor}
Since $q^n \leq \mathrm{deg}(\Psi_\a^n) \leq (d+1) q^n$, it follows
that 
\[
\lambda(\Psi_\alpha)=\lambda(\Psi_\alpha^{-1})=q.
\]
\end{cor}

To any polynomial automorphism $f=(f_0,f_1,f_2)$ of $\mathbb{C}^3$
of degree $d$ one can associate a birational self map of $\mathbb{P}^3_\mathbb{C}$ 
as follows
\[
(z_0:z_1:z_2:z_3)\dashrightarrow\left(z_3^df_0\left(\frac{z_0}{z_3},
\frac{z_1}{z_3},\frac{z_2}{z_3}\right):z_3^df_1\left(\frac{z_0}{z_3},
\frac{z_1}{z_3},\frac{z_2}{z_3}\right):z_3^df_2\left(\frac{z_0}{z_3},
\frac{z_1}{z_3},\frac{z_2}{z_3}\right):z_3^d\right);
\]
we still denote it by $f$.

If $g=(g_0:g_1:g_2:g_3)$ is a birational self map of $\mathbb{P}^3_\mathbb{C}$, 
the \textit{indeterminacy set} $\mathrm{Ind}(g)$ of $g$ is the set where~$g$ 
is not defined, that is 
\[
\mathrm{Ind}(g)=\big\{m\in\mathbb{P}^3_\mathbb{C}\,\vert\, g_0(m)=g_1(m)
=g_2(m)=g_3(m)=0\big\}.
\]
Remark that if we look at a birational map of $\mathbb{P}^3_\mathbb{C}$
that comes from a polynomial automorphism of $\mathbb{C}^3$ then its
indeterminacy set is contained in $\{z_3=0\}$.

A polynomial automorphism of $\mathbb{C}^3$ is \textit{algebraically stable} 
if for every $n \in \N$, 
\[
f^n\big(\{z_3=0\} \smallsetminus \mathrm{Ind}(f^n)\big) \not\subset \mathrm{Ind}(f).
\]
Let us recall the following result. 

\begin{prop}[\cite{FornaessSibony}]\label{Pro:algebraicstability}
The map $f$ is algebraically stable if and only if 
$\mathrm{deg}(f^n)=(\mathrm{deg}(f))^n$ for every $n \geq 1$.
\end{prop}

Lemma \ref{degpsi} and Proposition \ref{Pro:algebraicstability} imply
that $\Psi_\alpha$ is not algebraically stable, as well as $\Psi_\a^{-1}$.
It can also be seen directly from the definition. Indeed the map
\[
\Psi_\a=\big((z_0+z_1)z_3^{q+d-1}+z_0^q z_2^d:z_0 z_3^{q+d-1}:\a z_2 z_3^{q+d-1}:z_3^{q+d}\big)
\]
sends $z_3=0$ onto $(1:0:0:0)$ and 
$\mathrm{Ind}(\Psi_\alpha)=\{z_0=0,\ z_3=0\}\cup\{z_2=0,\ z_3=0\}$. 
Similarly, we see that
\[
\mathrm{Ind}(\Psi_\a^{-1})=\{z_1=0,\ z_3=0\}\cup\{z_2=0,\ z_3=0\},
\]
and $\Psi_\a^{-1}$ sends $z_3=0$ onto 
$(0:1:0:0)\in\mathrm{Ind}(\Psi_\alpha^{-1})$.
On the other hand $\Phi_\a(\{z_0\neq 0,\ z_3=0\})=(1:0:0:0)$ does
not belong to $\mathrm{Ind}(\Phi_\a)$ and $(1:0:0:0)$ is 
a fixed point of $\Phi_\alpha$ hence 
$\Phi_\alpha^n(\{z_0\not=0,\,z_3=0\})=(1:0:0:0)$ for any $n\geq 1$. 
In particular, $\Phi_\a$ is algebraically stable.
We have for every $n \geq 0$, $\mathrm{deg}(\Phi_\a^n)=q^n$. 
Notice that for $n \geq 3$, there exist examples of maps $f \in \mathrm{Aut}(\C^3)$
which are algebraically stable but whose inverse $f^{-1}$ is not algebraically 
stable (let us mention the following example due to Guedj: 
$f=(z_0^2+\lambda z_1+az_2,\lambda^{-1}z_0^2+z_1,z_0)$ with $a$ and 
$\lambda$ in $\mathbb{C}^*$). Yet this is not the case for~$\Phi_\a$. Indeed, 
\[
\Phi_\a^{-1}(z_0:z_1:z_2:z_3)=\left(\frac{z_1z_3^{q-1}}{\a^l}:
-\frac{z_1 z_3^{q-1}}{\a^l}+\frac{z_0 z_3^{q-1}}{\a^l}-\frac{z_1^q}{\a^{lq}}:
\frac{z_2 z_3^{q-1}}{\a}:z_3^q\right)
\]
so $\Phi_\a^{-1}(\{z_1\neq 0,\ z_3=0\})=(0:1:0:0)$ does not belong to 
$\mathrm{Ind}(\Phi_\a^{-1})=\{z_1=0,z_3=0\}$ and is fixed 
by~$\Phi_\a^{-1}$. Hence $\Phi_\a^{-1}$ is also algebraically stable.
As a result one can state:

\begin{prop}
For any integer $n\geq 1$ the following equalities hold
\[
\deg (\Phi_\alpha^n)=\deg (\Phi_\alpha^{-n})=q^n,\qquad
\lambda(\Phi_\alpha)=\lambda(\Phi_\alpha^{-1})=q.
\]
\end{prop}

\section{Centralizers}

If $\mathrm{G}$ is a group and $f$ an element of $\mathrm{G}$, we denote 
by $\mathrm{Cent}(f,\mathrm{G})$ the centralizer of $f$ in~$\mathrm{G}$, 
that is
\[
\mathrm{Cent}(f,\mathrm{G}):=\big\{g\in\mathrm{G}\,\vert\, fg=gf\big\}.
\]

The description of centralizers of discrete dynamical systems is an 
important problem in real and complex dynamics: Julia (\cite{Julia})
and Ritt (\cite{Ritt}) showed that the centralizer of a rational function 
$f$ of $\mathbb{P}^1$ is in general the set of iterates of $f$ (we then 
say that the centralizer of $f$ is trivial) except for
some very special~$f$. Later Smale asked if the centralizer of a 
generic diffeomorphism of a compact manifold is trivial (\cite{Smale}). 
Since then a lot of mathematicians have looked at
this question in different contexts; for instance as recalled 
in~\S\ref{sec:intro} Lamy has proved that the centralizer of a
polynomial automorphism of $\mathbb{C}^2$ of H\'enon type is in general
trivial (\cite{Lamy}).

Fix $\alpha$ with $0<|\alpha|\leq 1$. We would like to describe $\mathrm{Cent}(\Phi_\alpha,\mathrm{Aut}(\C^3))$. 
Of course it contains $\big\{\Phi_\alpha^n\,\vert\,n\in\Z\big\}$ but also 
the following one-parameter family
\[
\big\{(\eta z_0,\eta z_1,\nu z_2)\,\vert\, \nu \in \C^{*},
\eta\text{ a $(q-1)$-th root of unity}\big\}.
\]

We show that the centralizer is essentially reduced to the 
iterates of $\Phi_\a$ and such maps. Since the automorphism 
$\phi_\a=\alpha^l (z_1 + z_0 + z_0^q, z_0)$ is of Hénon 
type, it follows from a result of Lamy \cite{Lamy} that  
$\mathrm{Cent}(\phi_\alpha,\mathrm{Aut}(\C^2)) \simeq \Z \rtimes \Z_{n}$
for some $n \in \N$.\\

Let $f \in \mathrm{Cent}(\Phi_\alpha,\mathrm{Aut}(\C^3))$; we write 
$f=(f_0,f_1,f_2)$. 

\begin{lem}
We have $\frac{\partial f_2}{\partial z_0}=0$, $\frac{\partial f_2}{\partial z_1}=0$. 
Therefore, the last component $f_2$ only depends on $z_2$, and in fact it is a 
homothety: 
\[
f_2(z_0,z_1,z_2)= f_2(z_2)=\mu z_2,\quad\quad \mu \in \C^*.
\]
\end{lem}

\begin{proof}
If we focus on the third coordinate in relation 
$\Phi_\alpha\circ f=f\circ\Phi_\alpha$, we get 
$\alpha f_2 = f_2 \circ \Phi_\a$, that is, for every $(z_0,z_1,z_2) \in\C^3$, 
\[
\alpha f_2(z_0,z_1,z_2) = f_2(\alpha^l z_0 + \alpha^l z_1 + \alpha^l z_0^q,
\alpha^l z_0,\alpha z_2).
\]
Taking the derivatives in the different coordinates, we obtain:
\begin{equation}\label{deriv}
\left\{
\begin{array}{rcl}
\a \frac{\partial f_2}{\partial z_0} & = & \a^l (1+ q z_0^{q-1}) 
\frac{\partial f_2}{\partial z_0} \circ \Phi_\a + \a^l 
\frac{\partial f_2}{\partial z_1} \circ \Phi_\a,\\
\a \frac{\partial f_2}{\partial z_1} & = & \a^l 
\frac{\partial f_2}{\partial z_0} \circ \Phi_\a,\\
\a \frac{\partial f_2}{\partial z_2} & = & \a 
\frac{\partial f_2}{\partial z_2} \circ \Phi_\a.
\end{array}
\right.
\end{equation}

Let us consider the first coordinate, and assume that 
$\frac{\partial f_2}{\partial z_0}\neq 0$; we will get a contradiction 
by looking at highest-order terms in $z_0$. Since $f_2 \in \C[z_1,z_2][z_0]$, 
we can write $f_2(z_0,z_1,z_2)=\sum\limits_{k \leq k_0} R_k(z_1,z_2) z_0^k$, 
where the $R_k$ are polynomials and $k_0$ is the degree in $z_0$ of $f_2$. 
From our hypothesis, $k_0 \geq 1$. We also look at the expansion of $R_{k_0}\neq 0$:
\[
R_{k_0}(z_1,z_2)=\sum\limits_{m \leq m_0} Q_m(z_2) z_1^m,\quad Q_{m_0} \neq 0.
\]
For the three terms, we look at the term of highest order in $z_0$:
\[
\left\{
\begin{array}{rcl}
\a \frac{\partial f_2}{\partial z_0} (z_0,z_1,z_2) & = & \a k_0 R_{k_0} (z_1,z_2) 
z_0^{k_0-1} + \dots\\
\a^l (1+ q z_0^{q-1}) \frac{\partial f_2}{\partial z_0} \circ \Phi_\a (z_0,z_1,z_2) 
& = & q k_0 \a^{l(k_0+m_0)} Q_{m_0} (\a z_2) z_0^{q k_0 + m_0 -1}+\dots\\
\a^l \frac{\partial f_2}{\partial z_1} \circ \Phi_\a(z_0,z_1,z_2)  & = & m_0 
\a^{l(k_0+m_0)} Q_{m_0} (\a z_2) z_0^{q k_0 + m_0 -1}+\dots
\end{array}
\right.
\]
Since we assume $k_0 \geq 1$, and $q >1$, we have $q k_0+m_0-1>k_0-1$ so 
the coefficient of the term in $z_0^{q k_0 + m_0 -1}$ must vanish. But this 
coefficient is $(q k_0 + m_0) \a^{l(k_0+m_0)} Q_{m_0} (\a z_2) \neq 0$, a 
contradiction.
Hence $\frac{\partial f_2}{\partial z_0}= 0$, and it follows from the second 
equation of (\ref{deriv}) that $\frac{\partial f_2}{\partial z_1}= 0$ as well. 
Therefore, $f_2=f_2(z_2)$.

Now, since $f \in \mathrm{Aut}(\C^3)$, 
we know that $f_2$ is of degree at most $1$. The map $f$ commutes 
with $\Phi_\a$, so it must preserve its fixed point $0_{\C^3}$, and we conclude that 
$f_2\colon z_2 \mapsto \mu z_2$ for some $\mu \in \C^*$. 
\end{proof} 

Recall that $\phi_\a\colon(z_0,z_1) \mapsto \alpha^l (z_1 + z_0 + z_0^q, z_0)$. 
Let us denote $\widetilde f := (f_0,f_1)$. By projecting the commutation relation
 on the first two coordinates, we get 
\begin{equation}\label{commmm}
\phi_\a \circ \widetilde f = \widetilde{f} \circ \Phi_\a.
\end{equation}

\begin{lem}
The map $\widetilde f$ only depends on the first two coordinates.
\end{lem}

\begin{proof}
We rewrite (\ref{commmm}) as the following system:
\begin{equation}\label{systt}
\left\{
\begin{array}{rcl}
\alpha^l f_0 + \alpha^l f_1 +\alpha^l f_0^q & = & f_0 \circ \Phi_\a\\
\alpha^l f_0 & = & f_1 \circ \Phi_\a.
\end{array}
\right.
\end{equation}
We then get:
\[
\alpha^l f_0 \circ \Phi_\a + \alpha^{2l} f_0 +\alpha^l f_0^q \circ \Phi_\a 
 = f_0 \circ \Phi_\a^2.
\]
Let $d_0$ be the degree of $f_0 \in \C[z_0,z_1] [z_2]$. Since $\Phi_\a$ does not 
change the degree in $z_2$, we obtain
\[
\mathrm{deg}(\alpha^l f_0 \circ \Phi_\a)=\mathrm{deg}(\alpha^{2l} f_0)=
\mathrm{deg}(f_0 \circ \Phi_\a^2)=d_0,\quad \mathrm{deg}(\alpha^l f_0^q 
\circ \Phi_\a)=d_0^q,
\]
but $q>1$, which implies that $d_0=0$: $f_0$ does not depend on $z_2$. Using 
the second equation of (\ref{systt}), we see that $f_1$ does not depend on 
$z_2$ either. 
\end{proof}

Therefore, Equation (\ref{commmm}) can be rewritten:
\[
\phi_\a \circ \widetilde f = \widetilde{f} \circ \phi_\a.
\]
But $\phi_\a$ is a Hénon automorphism, so according to \cite{Lamy} one has:

\begin{cor}
The map $\widetilde f$ belongs to the countable set 
$\mathrm{Cent}(\phi_\alpha,\mathrm{Aut}(\C^2)) \simeq \Z \rtimes \Z_{n}$, $n \in \N$.
\end{cor}

We have seen that for any 
$f=(f_0,f_1,f_2) \in \mathrm{Cent}(\Phi_\alpha,\mathrm{Aut}(\C^3))$,
$(f_0,f_1)$ depends only on $(z_0,z_1)$ and belongs to 
$\mathrm{Cent}(\phi_\alpha,\mathrm{Aut}(\C^2))$, and that 
$f_2$ depends only on $z_2$ and is a homothety. We conclude:
\begin{prop}\label{Pro:commutateur}
The centralizer of $\Phi_\alpha$ in $\mathrm{Aut}(\C^3)$ is uncountable.
More precisely
\[
\mathrm{Cent}(\Phi_\alpha,\mathrm{Aut}(\C^3)) = \mathrm{Cent}(\phi_\alpha,\mathrm{Aut}(\C^2)) \times \big\{z_2 \mapsto \mu z_2\ \vert\ \mu \in \C^*\big\} \simeq (\Z \rtimes \Z_{n}) \times \C^*,\ n\in\N.
\]
\end{prop}

\begin{cor}
The centralizer of $\Psi_\alpha$ in $\mathrm{Aut}(\C^3)$ is uncountable.
\end{cor}

\section{Dynamics on the invariant hypersurface $z_2=0$}

The following holds for any $\alpha \neq 0$.
Let us recall that the Fibonacci sequence is the sequence $(F_n)_n$ defined
by: $F_0=0$, $F_1=1$ and for all $n\geq 2$
\[
F_n=F_{n-1}+F_{n-2}.
\]

The hypersurface $\{z_2=0\}$ is invariant, and when $|\a|<1$, 
it attracts every point $p \in \C^3$. On restriction to this hypersurface, 
the growth is given by the Fibonacci numbers $(F_n)_n$:
\begin{equation}\label{fibo}
\restriction{\Psi_\a^n}{z_2=0}=(F_{n+1} z_0 + F_{n} z_1, F_{n} z_0 + F_{n-1} z_1),\quad n\geq 1.
\end{equation}
Since
\[
\Psi_\a^{-1}(z_0,z_1,z_2)=\left(z_1,-z_1+z_0-z_1^q \frac{z_2^d}{\a^d},\frac{z_2}{\a}\right),
\]
similarly, we have
\begin{equation}\label{fibo2}
\restriction{\Psi_\a^{-n}}{z_2=0}\colon (z_0,z_1)\mapsto(-1)^n(F_{n-1} z_0 - F_{n} z_1, -F_{n} z_0 + F_{n+1} z_1),\quad n \geq 1.
\end{equation}
Moreover, it is easy to see that any periodic point of $\Psi_\a$ 
belongs to the hypersurface $\{z_2=0\}$. In fact, $\Psi_\a$ 
has a unique fixed point at finite distance, $0_{\C^3}=(0,0,0)$, and 
has no periodic point of period larger than $1$. Let 
$\varphi:=\frac{1+\sqrt{5}}{2}$ be the golden ratio 
and $\varphi':=-1/\varphi$. Since 
\[
F_n=\frac{\varphi^n-(\varphi')^n}{\sqrt{5}}=\frac{\varphi^n}{\sqrt{5}} 
+ o(1),
\]
we deduce from (\ref{fibo}) that any point $\left(\varphi' z,z,0\right)$ with 
$z \in \C$ converges to $0_{\C^3}$ when we iterate $\Psi_\a$, while 
any other point of the form $\left(\beta z,z,0\right)$ with $z \neq 0$ and 
$\beta \neq \varphi'$ goes to infinity. Likewise, we see from 
(\ref{fibo2}) that any point $\left(\varphi z,z,0\right)$ with $z \in \C$ 
converges to $0_{\C^3}$ when we iterate $\Psi_\a^{-1}$, while any other 
point of the form $\left(\beta z,z,0\right)$ with $z \neq 0$ and 
$\beta \neq \varphi$ goes to infinity. Furthermore, in both cases, the
 speed of the convergence is exponential since it is in 
$O(|\varphi|^{-n})$ with $|\varphi| >1$. In other terms, the linear map 
$\restriction{\Psi_\a}{z_2=0}\colon (z_0,z_1) \mapsto (z_0+z_1,z_0)$ is 
hyperbolic, with a unique fixed point $0_{\C^2}=(0,0)$ of saddle type, and 
whose stable, respectively unstable manifolds correspond to the 
following lines:
\[
W_{\restriction{\Psi_\a}{z_2=0}}^s(0_{\C^2})=\Delta_{\varphi'}:=\left\{\left(\varphi' z,z\right)\,\big|\, z\in \C\right\},\quad 
W_{\restriction{\Psi_\a}{z_2=0}}^u(0_{\C^2})=\Delta_{\varphi}:=\left\{\left(\varphi z,z\right)\,\big|\, z \in \C\right\}.
\]
Moreover, $\varphi$ and $\varphi'$ are just the eigenvalues of 
$\restriction{\Psi_\a}{z_2=0}$, and $\Delta_{\varphi}$, $\Delta_{\varphi'}$ 
the corresponding eigenspaces.

\section{Points with bounded forward orbit, description of the stable manifold $W_{\Psi_\a}^s(0_{\C^3})$}\label{Section:escapingspeed}

When $0<|\alpha|<1$, we remark that $0_{\C^3}$ is a 
hyperbolic fixed point of saddle type. The tangent space at $0_{\C^3}$ can be written as 
$T_{0_{\C^3}}(\C^3)=E_{\Psi_\a}^s(0_{\C^3}) \oplus E_{\Psi_\a}^u(0_{\C^3})$, where the 
stable, respectively unstable spaces are given by
\[
E_{\Psi_\a}^s(0_{\C^3})= \Delta_{\varphi'} \times \{0\} \oplus \{0_{\C^2}\} \times \C,\quad \quad
E_{\Psi_\a}^u(0_{\C^3})= \Delta_{\varphi} \times \{0\}.
\]

These spaces integrate to stable and unstable 
manifolds 
\[
W_{\Psi_\a}^s(0_{\C^3}):=\big\{p \in \C^3\ \vert \ \lim\limits_{n \to +\infty}\Psi_\a^n(p)=0_{\C^3}\big\},\quad
W_{\Psi_\a}^u(0_{\C^3}):=\big\{p \in \C^3\ \vert \ \lim\limits_{n \to +\infty}\Psi_\a^{-n}(p)=0_{\C^3}\big\}
\]
which are invariant by the dynamics; furthermore, $W_{\Psi_\a}^u(0_{\C^3})=\Delta_{\varphi} \times \{0\}$, 
while $W_{\Psi_\a}^s(0_{\C^3})$ is biholomorphic to $\C^2$ (see \cite{Sibony}). Note that
$\big(\Delta_{\varphi'}\times\{0\}\big)\cup\big(\{0_{\C^2}\}\times\C\big)\subset W_{\Psi_\a}^s(0_{\C^3})$, but 
it is easy to see
\footnote{Indeed if $p=(p_0,p_1,p_2)$ satisfies $p_2 \neq 0$ 
 and $p_1 = - \varphi p_0$, we see that
$P_\a^{(0)}(p) + \varphi P_\a^{(1)}(p)=p_0 (1 + \varphi - \varphi^2) + \varphi p_0^q p_2^d
 = \varphi p_0^q p_2^d \neq 0$ hence $\Delta_{\varphi'} \times \C$ is not 
 left invariant 
 by $\Psi_\a$.} that 
 $W_{\Psi_\a}^s(0_{\C^3}) \neq \Delta_{\varphi'} \times \C$.
 
\begin{figure}[H]
\begin{center}
\includegraphics [width=10.5cm]{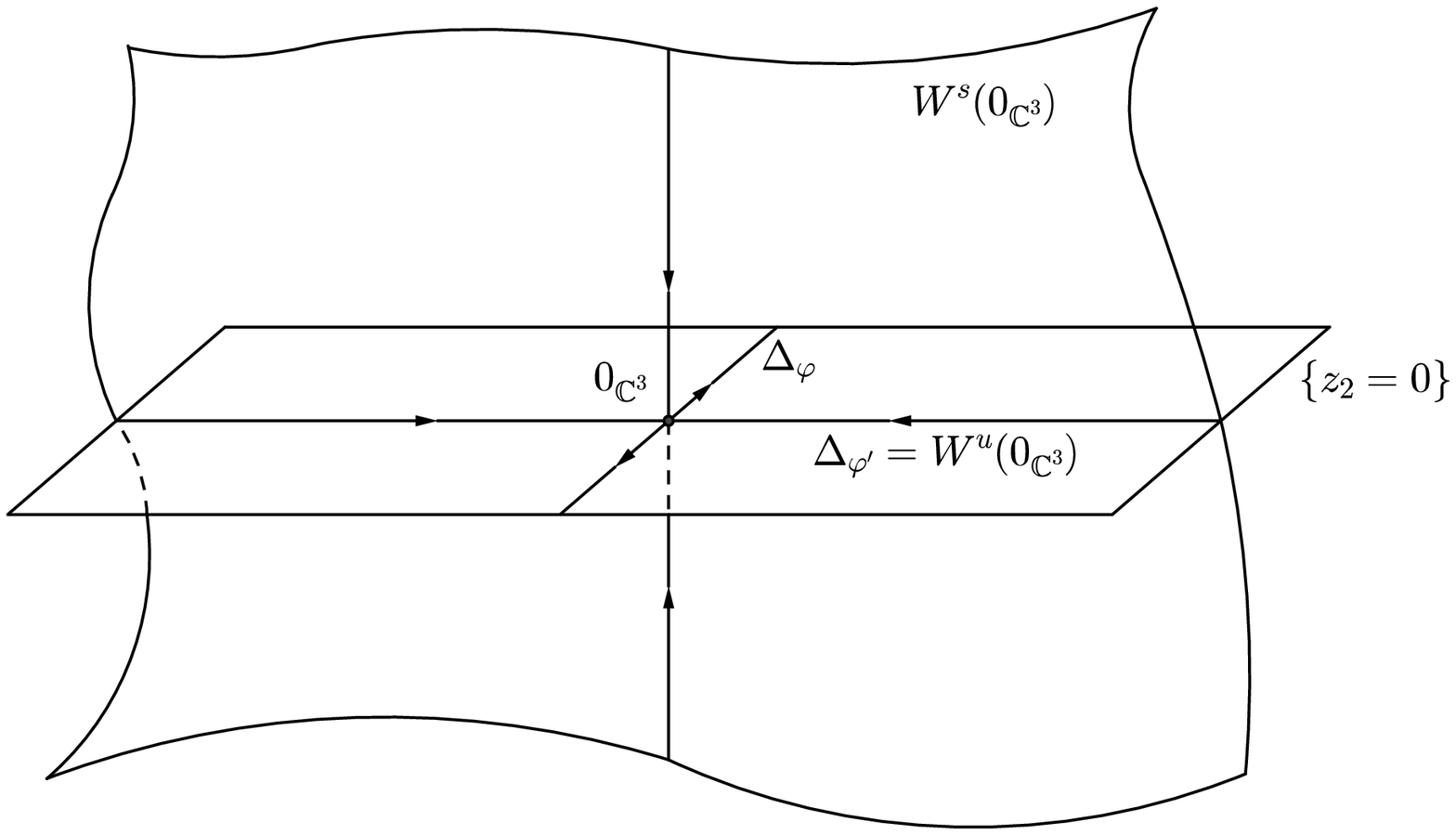}
\end{center}
\end{figure}

In the next statement, we introduce a series that encodes the growth of 
forward iterates of a point.

\begin{lem}\label{lemmarecphi}
Let $p=(p_0,p_1,p_2) \in \C^3$. For every $n \geq 0$ and any $\alpha\in \C$, 
we have
\begin{equation}
P_\a^{(n+1)}(p) + \varphi^{-1} P_\a^{(n)}(p) = \varphi^n \left(\varphi p_0+p_1+p_2^d
\sum\limits_{j=0}^{n} \left(P_\a^{(j)}(p)\right)^q \varphi^{-j} \alpha^{jd}\right).
\end{equation}
\end{lem}

\begin{proof}
For $n \geq 0$ we have the following set of equalities:
\[
\begin{array}{lrcccccl}
& P_\a^{(n+1)}(p) &=& P_\a^{(n)}(p) & + & P_\a^{(n-1)}(p) & + & (P_\a^{(n)}(p))^q (\alpha^n p_2)^d\\
(\times\ \varphi) & P_\a^{(n)}(p) &=& P_\a^{(n-1)}(p) & + &  P_\a^{(n-2)}(p) & + & (P_\a^{(n-1)}(p))^q (\alpha^{n-1} p_2)^d\\
(\times\ \varphi^2) & P_\a^{(n-1)}(p) &=& P_\a^{(n-2)}(p) & + &  P_\a^{(n-3)}(p) & + & (P_\a^{(n-2)}(p))^q (\alpha^{n-2} p_2)^d\\
 & \vdots &=& \vdots & + &  \vdots & + & \vdots \\
 (\times\ \varphi^{n-1}) & P_\a^{(2)}(p) &=& P_\a^{(1)}(p) & + &  p_0 & + & (P_\a^{(1)}(p))^q (\alpha p_2)^d\\
 (\times\ \varphi^{n}) & P_\a^{(1)}(p) &=& p_0 & + &  p_1 & + & (P_\a^{(0)}(p))^q (p_2)^d \\
\end{array} 
\]
Summing up, and because $\varphi^2-\varphi-1=0$, we obtain
\[
P_\a^{(n+1)}(p) + \varphi^{-1}  P_\a^{(n)}(p) = \varphi^n \left(\varphi p_0 + p_1 + p_2^d \sum\limits_{j=0}^{n} 
\left(P_\a^{(j)}(p)\right)^q \varphi^{-j} \alpha^{jd}\right).
\]
\end{proof}

For every $n \geq -1$, we define the polynomial $g_n \in \C[z]=\C[z_0,z_1,z_2]$ by 
\[
g_n(z):=\varphi z_0 + z_1 + z_2^d \sum\limits_{j=0}^{n} \left(P_\a^{(j)}(z)\right)^q \varphi^{-j} \alpha^{jd}=
\big(P_\a^{(n+1)}(z) + \varphi^{-1} P_\a^{(n)}(z)\big)\varphi^{-n} .
\] 
We also introduce the power series 
\[
g(z):=\varphi z_0 + z_1 + z_2^d \sum\limits_{j=0}^{+\infty} 
\left(P_\a^{(j)}(z)\right)^q \varphi^{-j} \alpha^{jd}=\varphi z_0+z_1+\sum_{j=-1}^{+\infty}\varphi^{-(j+1)}
\big(P_{\a}^{(j+2)}(z)-P_{\a}^{(j+1)}(z)-P_{\a}^{(j)}(z)\big).
\]
Let us denote by $\mathcal{D}$ its domain of definition, that is the set 
of $p \in \C^3$ such that the series
$\sum_{j}\left(P_\a^{(j)}(p)\right)^q \varphi^{-j} \alpha^{jd}$ converges, and let 
\[
\mathcal{Z}:=\big\{p \in \mathcal{D}\ \vert \ g(p)=0\big\}
\]
be the set of its zeroes. It is easy to check 
that both $\mathcal{D}$ and $\mathcal{Z}$ are invariant by the dynamics, 
that is $\Psi_\a(\mathcal{D}) \subset \mathcal{D}$ and 
$\Psi_\a(\mathcal{Z}) \subset \mathcal{Z}$. Moreover, if $p \in \mathcal{D}$, we 
denote by 
$r_n(p):=\sum\limits_{j\geq n+1}\left(P_\a^{(j)}(p)\right)^q \varphi^{-j} \alpha^{jd}$ 
the tail of the corresponding series.

\begin{cor}\label{unbounded}
Suppose $0 < |\alpha| \leq 1$. 
Let $K^+_{\Psi_\alpha}$ denote the set of points $p = (p_0,p_1,p_2)$ whose forward orbit 
$\{\Psi_\a^n(p),\ n \geq 0\}$ is bounded. 
This is equivalent to the fact that the sequence $\big(|P_\a^{(n)}(p)|\big)_{n\geq 0}$ is bounded.
If $p \in K^+_{\Psi_\alpha}$, then for every $n \geq 0$, we have 
\begin{equation}\label{equapprochee0}
|g_n(p)| = O(\varphi^{-n}).
\end{equation}
In particular we deduce that 
\begin{equation}\label{bornezero0}
K^+_{\Psi_\alpha} \subset \mathcal{Z},\quad \text{and}\quad |r_n(p)| = O(\varphi^{-n}).
\end{equation}
\end{cor}

\begin{proof}
It follows immediately from Lemma \ref{lemmarecphi}. 
Indeed under our assumptions we have:
\[
|g_n(p)|\leq (|P_\a^{(n+1)}(p)| + \varphi^{-1} |P_\a^{(n)}(p)|)\varphi^{-n} = O(\varphi^{-n}).
\]
This implies $p \in \mathcal{Z}$. Then we also have $g_n(p)=g(p)-p_2^d r_n(p)=-p_2^d r_n(p)$ and $|r_n(p)| = O(\varphi^{-n})$.
\end{proof}

\begin{rem}
We can see (\ref{bornezero0}) as a codimension one condition 
that points with bounded forward orbit have to satisfy. 
Also, we see from (\ref{equapprochee0}) that locally, such points are close 
to the analytic manifold $\mathcal{Z}_n:=\big\{p \in \C^3\ \vert \ g_n(p)=0\big\}$ 
for $n \geq 0$ big. 
If $p=(p_0,p_1,0)$, we recover from (\ref{bornezero0})  
that $p$ has bounded forward orbit if and only if it belongs to the stable manifold
$W_{\Psi_\a}^s(0_{\C^3}) \cap \{z_2=0\} = \Delta_{\varphi'} \times \{0\}$.
\end{rem}

When $0<|\alpha|<1$, we have the following analytic characterization of the stable manifold $W_{\Psi_\a}^s(0_{\C^3})$.

\begin{prop}\label{stablemanifold}
Assume $0<|\alpha|<1$. 
The point $p=(p_0,p_1,p_2) \in \C^3$ belongs to the stable manifold
 $W_{\Psi_\a}^s(0_{\C^3})$ if and only if the following properties hold:
\smallskip
\begin{itemize}
\item $p \in \mathcal{Z}$;
\smallskip
\item the series $\displaystyle\sum_j |r_j(p)| \varphi^j$ is convergent.
\end{itemize}
Equivalently, $p\in W_{\Psi_\a}^s(0_{\C^3})$ if and only if $\sum_j |g_j(p)| \varphi^j$ converges.
\end{prop}

\begin{proof}
If $p$ belongs to the stable manifold $W_{\Psi_\a}^s(0_{\C^3})$, then its forward
orbit is bounded and Corollary \ref{unbounded} tells us that 
$p \in \mathcal{Z}$.
Moreover we have 
\[
|r_n(p)|= O\left(\sum\limits_{j\geq n+1} \varphi^{-j} \alpha^{jd}\right)=O(\varphi^{-n} |\alpha|^{nd}),
\]
hence $|r_n(p)| \varphi^n = O(|\a|^{nd})$ and the series $\sum_j |r_j(p)| \varphi^j$ converges.

For the other implication, we get from Lemma \ref{lemmarecphi} that for every $j \geq 0$,
\begin{equation}\label{eqstab}
P_\a^{(j+1)}(p) + \varphi^{-1} P_\a^{(j)}(p) = - p_2^d \varphi^j r_j(p).
\end{equation}
Now let $n \geq 0$. Write equations (\ref{eqstab}) for $j = 0, \dots,n$ and 
combine them to obtain 
\[
P_\a^{(n+1)}(p)=\frac{(-1)^{n+1}}{\varphi^{n+1}} p_0+ p_2^d
\sum\limits_{j=0}^{n} (-1)^{n+j+1} r_j(p) \varphi^j \varphi^{j-n}.
\]
The first term of the right hand side goes to $0$ with $n$; we split the sum 
as follows:
\begin{align*}
\sum\limits_{j=0}^{n} (-1)^{n+j+1} r_j(p) \varphi^j \varphi^{j-n}&=
\sum\limits_{j=0}^{\lfloor n/2 \rfloor} (-1)^{n+j+1} r_j(p) \varphi^j \varphi^{j-n}\\
&+ \sum\limits_{j=\lfloor n/2 \rfloor+1}^{n} (-1)^{n+j+1} r_j(p) \varphi^j \varphi^{j-n}.
\end{align*}
We get
\[
\left|\sum\limits_{j=0}^{\lfloor n/2 \rfloor} (-1)^{n+j+1} r_j(p) \varphi^j \varphi^{j-n}\right|
\leq \varphi^{\lfloor n/2 \rfloor-n} \sum\limits_{j=0}^{+\infty} |r_j(p)| \varphi^j 
\]
hence it goes to $0$ with respect to $n$. For the remaining term, we estimate
\[
\left|\sum\limits_{j=\lfloor n/2 \rfloor+1}^{n} (-1)^{n+j+1} r_j(p) \varphi^j \varphi^{j-n}\right|
\leq \sum\limits_{j=\lfloor n/2 \rfloor+1}^{+\infty} |r_j(p)| \varphi^j,
\]
which goes to $0$ as well. We conclude that $\lim\limits_{n \to +\infty} P_\a^{(n)}(p) = 0$, 
hence $p \in W_{\Psi_\a}^s(0_{\C^3})$. 

The other equivalence follows from the fact that for $p \in \mathcal{Z}$, we have $g_n(p)=-p_2^d r_n(p)$. 
\end{proof}

\begin{cor}\label{corvarstable}
Assume $0<|\alpha|<1$. Then 
the forward orbit of a point $p=(p_0,p_1,p_2)$ is bounded if and only if $p$ belongs to
 the stable manifold $W_{\Psi_\a}^s(0_{\C^3})$; in other terms, $K^+_{\Psi_\alpha}=W_{\Psi_\a}^s(0_{\C^3})$.
\end{cor}

\begin{proof}
Let $p \in K^+_{\Psi_\alpha}$. We have already seen in Corollary \ref{unbounded} that 
$p \in \mathcal{Z}$. Moreover,  if we denote 
$r_n(p):=\sum\limits_{j\geq n+1}(P_\a^{(j)}(p))^q \varphi^{-j} \alpha^{jd}$, then
$\displaystyle\sum_j |r_j(p)| \varphi^j$ is convergent since 
$|r_j(p)|=O(\varphi^{-j} |\alpha|^{jd})$. Then, Proposition \ref{stablemanifold} tells us 
that $p \in W_{\Psi_\a}^s(0_{\C^3})$. The other implication is straightforward. 
\end{proof}

\begin{lem}\label{lemalpha1}
Assume now $|\alpha|=1$. We have seen that $p \in K_{\Psi_\a}^+$
 implies that $p \in \mathcal{Z}$ 
and $|r_n(p)| = O(\varphi^{-n})$. Conversely, if $p \in \mathcal{Z}$ 
and $\sum_j |r_j(p)|\varphi^{j}$ is convergent, then $p \in K_{\Psi_\a}^+$.
\end{lem}

\begin{proof}
Assume that $p \in \mathcal{Z}$ and that $\displaystyle\sum_j |r_j(p)|\varphi^{j}$ 
converges. As previously, we have for every $n \geq 0$:
\[
P_\a^{(n)}(p)=\frac{(-1)^{n}}{\varphi^{n}} p_0+ p_2^d\sum\limits_{j=0}^{n-1} (-1)^{n+j} r_j(p) \varphi^j \varphi^{j-(n-1)}.
\]
From our assumption we deduce that 
$|P_\a^{(n)}(p)|\leq |p_0|+|p_2|^d \sum\limits_{j=0}^{+\infty} |r_j(p)| \varphi^j$, 
hence $|P_\a^{(n)}(p)|=O(1)$.
\end{proof}

\section{Birational conjugacy, dynamical properties of automorphisms of Hénon type}\label{sec:lost}

Assume $0<|\alpha|\leq 1$. 
We recall here some facts concerning the dynamics of the H\'enon automorphism 
$\phi_\a$, and so on the dynamics of $\Phi_\alpha=(\phi_\alpha,\alpha z_2)$. Denote by $F^+_{\phi_\alpha}$ 
the largest open set on which $(\phi_\alpha^n)_n$ is locally equicontinuous, 
by $K^+_{\phi_\alpha}$ the set of points $p \in \C^2$ such that $(\phi_\alpha^n(p))_{n \geq 0}$ is 
bounded and by $J^+_{\phi_\alpha}$ its topological boundary. The automorphism $\phi_\a$ is \textit{regular}, 
that is, $\mathrm{Ind}(\phi_\a)\cap \mathrm{Ind}(\phi_\a^{-1})= \emptyset$. In particular, it is 
algebraically stable, hence the following limit 
exists and defines a Green function:
\[
G^+_{\phi_\a}(z_0,z_1):=\lim_{n\to +\infty}\frac{\log^+\vert\vert\phi_\a^n(z_0,z_1)\vert\vert}{q^n}.
\]
It satisfies the invariance property 
$G^+_{\phi_\a}\circ\phi_\alpha=q \cdot G^+_{\phi_\a}$. We define the associate current $T^+_{\phi_\a}=\mathrm{dd}^cG^+_{\phi_\a}$, where 
$\mathrm{d}^c=\frac{\mathbf{i}(\overline{\partial}-\partial)}{2\pi}$. Of course there
are similar objects $F^-_{\phi_\alpha}$, $K^-_{\phi_\alpha}$, $J^-_{\phi_\alpha}$,
$G^-_{\phi_\a}$ and $T^-_{\phi_\a}$ associated to the 
inverse map $\phi_\alpha^{-1}$;
we also set $K_{\phi_\alpha}:=K^+_{\phi_\alpha}\cap K^-_{\phi_\alpha}$.
One inherits a probability measure 
$\mu_{\phi_\a}=T^+_{\phi_\a}\wedge T^-_{\phi_\a}$ which is invariant by $\phi_\a$. 
Set $\phi:=\phi_1$. If $|\alpha|>1$, then 
$\phi_{\alpha}=\alpha^l \phi = ((\alpha^{-1} )^{l}\phi^{-1})^{-1}$, 
where $\alpha^{-1}<1$ and $\phi^{-1}$ is of the same 
form as $\phi$, so that we can restrict ourselves to the case where $|\alpha|\leq 1$.   Besides, according to
\cite{BedfordSmillie:JAMS, FornaessSibony, BedfordSmillie, 
BedfordLyubichSmillie:1,BedfordLyubichSmillie:2} the following properties hold: for $0<\a\leq 1$,
\smallskip
\begin{itemize}
\item the function $G^+_{\phi_\a}$ is H\"older continuous;
\smallskip
\item we have the following characterization of points with bounded orbit:
\begin{equation}\label{bddorbits}
K^\pm_{\phi_\a}=\big\{p \in \C^2 \,\vert\,G^\pm_{\phi_\a}(p)=0\big\};
\end{equation}
this tells us that points either have bounded foward orbit, or escape to infinity 
with maximal speed;
\item let $p$ be a saddle point of $\phi_\a$, then $J^+_{\phi_\a}$ is 
the closure of the stable manifold $W_{\phi_\a}^s(p)$;
\smallskip

\item the support of $T_{\phi_\a}^+$ coincides with the boundary of
$K^+_{\phi_\a}$, that is $J^+_{\phi_\a}$;
\smallskip

\item the current $T^+_{\phi_\a}$ is extremal among positive closed
currents in $\mathbb{C}^2$ and is -- up to a multiplicative constant -- 
the unique positive closed current supported on $K^+_{\phi_\a}$;
\smallskip 

\item the measure $\mu_{\phi_\a}$ has support in the compact set $\partial K_{\phi_\a}$,
is mixing, maximises entropy and is well approximated by Dirac 
masses at saddle points.
\end{itemize}
\smallskip

One introduces analogous objects for the automorphism $\Phi_\a$.
In particular, $\Phi_\a$ is algebraically stable so we can also define the Green function 
\[
G_{\Phi_\a}^+:=\displaystyle\lim_{n\to +\infty}\frac{\log^+\vert\vert\Phi_\alpha^n
\vert\vert}{q^n}.
\]
In fact, for any $(z_0,z_1,z_2) \in \C^3$, 
$G_{\Phi_\a}^+(z_0,z_1,z_2)=G_{\phi_\a}^+(z_0,z_1)$ 
because if $z_2 \neq 0$,
$\lim\limits_{n\to +\infty} \frac{\log|\alpha^n z_2|}{q^n} = 0$. 
We introduce the holomorphic map 
\[
h\colon \C^3 \to \C^2,\qquad  h\colon(z_0,z_1,z_2) \mapsto (z_0 z_2^l,z_1 z_2^l).
\] 
It follows from the previous remark that $G_{\Phi_\a}^+\circ \theta=G_{\phi_\a}^+ \circ h$
where $\theta=(z_0z_2^l,z_1z_2^l,z_2)$ conjugates $\Phi_\alpha$ to $\Psi_\alpha$ (\emph{i.e.}
$\theta\Psi_\alpha=\Phi_\alpha\theta$). Moreover, for $0<|\alpha|<1$, one gets that 
\[
K^+_{\Phi_\alpha}=K^+_{\phi_\alpha}\times\mathbb{C}, \qquad\qquad 
K^-_{\Phi_\alpha}=K^-_{\phi_\alpha}\times\{0\},\qquad\qquad  K_{\Phi_\alpha}=K_{\phi_\alpha}\times\{0\},
\]
while for $|\alpha|=1$, 
\[
K^+_{\Phi_\alpha}=K^+_{\phi_\alpha}\times\mathbb{C}, \qquad\qquad 
K^-_{\Phi_\alpha}=K^-_{\phi_\alpha}\times\{\C\},\qquad\qquad  K_{\Phi_\alpha}=K_{\phi_\alpha}\times\{\C\}.
\]

\smallskip

When $f$ is an algebraically stable polynomial automorphism of 
$\mathbb{C}^3$ one can inductively define the analytic sets $X_j(f)$ by 
\[
\left\{
\begin{array}{ll}
X_1(f)=\overline{f\big(\{z_3=0\}\smallsetminus\mathrm{Ind}(f)\big)} \\
X_{j+1}(f)=\overline{f\big(X_j(f)\smallsetminus\mathrm{Ind}(f)\big)}\quad\forall\,j\geq 1
\end{array}
\right.
\]
The sequence $(X_j(f))_j$ is decreasing, $X_j(f)$ is non-empty since $f$ is 
algebraically stable, so it is stationary. Denote by $X(f)$ the corresponding limit 
set. 
An algebraically stable polynomial automorphism $f$ of $\mathbb{C}^3$ is 
\textit{weakly regular} if $X(f)\cap\mathrm{Ind}(f)=\emptyset$. For instance
elements of $\mathcal{H}$ are weakly regular. A weakly regular automorphism is 
algebraically stable. Moreover $X(f)$ is an attracting set for $f$; in other words 
there exists an open neighborhood $\mathcal{V}$ of $X(f)$ such that 
$f(\mathcal{V})\Subset\mathcal{V}$ and 
$\displaystyle\bigcap_{j=1}^{+\infty} f^j(\mathcal{V})=X(f)$. 

Both $\Phi_\alpha$ and $\Phi_\alpha^{-1}$ are weakly regular; furthermore  
$X(\Phi_\a)=(1:0:0:0)$ and $X(\Phi_\a^{-1})=(0:1:0:0)$. Therefore 
$(1:0:0:0)$ (resp. $(0:1:0:0)$) is an attracting point for $\Phi_\alpha$ (resp. 
$\Phi_\alpha^{-1}$). From \cite[Corollary 1.8]{GuedjSibony} the basin of attraction 
of $(1:0:0:0)$ is biholomorphic to $\mathbb{C}^3$. 

As recalled above, 
the automorphisms $\phi_\a^\pm$ are regular, hence weakly regular, and similarly to $\Phi_\a^\pm$, 
they possess attractors 
$X(\phi_\a^+)=(1:0:0)$ and $X(\phi_\a^-)=(0:1:0)$ whose basins are biholomorphic to $\C^2$.

\medskip

We do not inherit these properties for $\Psi_\alpha$. Indeed  
remark that $K_{\Phi_\alpha}$, $X(\Phi_\a)$ and $X(\Phi_\a^{-1})$ 
are contained in $\{z_2=0\}$; but $\{z_2=0\}$ is 
contracted by $\theta^{-1}$ (recall that $\theta\Psi_\alpha=\Phi_\alpha\theta$)
onto $\{z_2=z_3=0\}$ and $\{z_2=z_3=0\}=\mathrm{Ind}(\Psi_\alpha)$.

\section{Definition of a Green function for $\Psi_\a$}

In this part, we assume $0<|\a|\leq 1$. As for $\phi_\a$ and $\Phi_\a$, and 
despite the fact that $\Psi_\a$ is not algebraically stable, we will see that it 
is possible to define a Green function for the automorphism $\Psi_\a$ which has 
almost as good properties. In particular, we will see that this function
carries a lot of information about the dynamics of the automorphism $\Psi_\a$. 

Let $p=(p_0,p_1,p_2)\in \C^3$, and define 
$C=C(p_2):=3 \max(1,|p_2|^d)>0$. 
We remark that for $n \geq 0$, we have $|\alpha^n p_2|^d \leq C$, hence
$\max(\|\Psi_\a^{n+1}(p)\|,1)\leq C \max(\|\Psi_\a^{n}(p)\|,1)^q$. We deduce that for every $n \geq 0$,
\[
\left|\frac{\log^+\|\Psi_\a^{n+1}(p)\|}{q^{n+1}}-\frac{\log^+\|\Psi_\a^{n}(p)\|}{q^{n}}\right|\leq \frac{\log (C)}{q^{n+1}},
\]
hence 
\[
\lim_{n\to +\infty}\frac{\log^+\vert\vert\Psi_\alpha^n(p)\|}{q^n}=\lim_{n\to +\infty}\frac{\log^+\max(|P_\a^{(n)}(p)|,|P_\a^{(n-1)}(p)|)}{q^n}=:G^+_{\Psi_\alpha}(p)
\] 
exists. By construction, the function $G_{\Psi_\a}^+$ satisfies 
$G_{\Psi_\a}^+\circ \Psi_\a = q \cdot G_{\Psi_\a}^+$. We note that on restriction to 
the hypersurface $\{z_2=0\}$,
\[
\restriction{G_{\Psi_\a}^+}{\{z_2=0\}}=0.
\] 
Indeed, if $p=(p_0,p_1,0)$, then $\log^+ |P_\a^{(n)}(p)| = O(n)$ since we have seen 
that the forward iterates of $p$ grow at most with Fibonacci speed.

For every $n \geq 0$, we have $\theta \circ \Psi_\a^n=\Phi_\a^n\circ \theta$. In 
particular, the following limits exist and satisfy:
\begin{equation}\label{eqconjgreen}
\displaystyle\lim_{n\to +\infty}\frac{\log^+\vert\vert\theta \circ \Psi_\alpha^n
\vert\vert}{q^n}=\lim_{n\to +\infty}\frac{\log^+\vert\vert\Phi_\alpha^n \circ \theta
\vert\vert}{q^n}=G_{\Phi_\a}^+\circ \theta.
\end{equation}
Define the open set $\mathcal{U}:=\big\{(z_0,z_1,z_2)\in\mathbb{C}^3\,\vert\,z_2\neq 0\big\}$ 
and let $p=(p_0,p_1,p_2) \in \mathcal{U}$. For every $n \geq 0$, recall that
\[
\theta \circ \Psi_\a^n(p)=\theta\Big(P_\a^{(n)}(p),P_\a^{(n-1)}(p),\a^n p_2\Big)=\Big(P_\a^{(n)}(p)(\a^n p_2)^l,P_\a^{(n-1)}(p)(\a^n p_2)^l,\a^n p_2\Big).
\]
For $j\in\{n-1,n\}$, we have 
\[
\log^+|P_\a^{(j)}(p)| -l \log^+ |\a^n p_2|\leq \log^+ |P_\a^{(j)}(p)(\a^n p_2)^l |\leq \log^+|P_\a^{(j)}(p)| + l \log^+ |\a^n p_2|,
\] 
so that 
\[
\frac{\log^+\vert\vert\theta \circ \Psi_\alpha^n(p)\vert\vert}{q^n}=\frac{\log^+\vert\vert\Psi_\alpha^n(p)\vert\vert}{q^n}+o(1).
\] 
We deduce from (\ref{eqconjgreen}) that
\begin{equation}\label{defgpsi}
G^+_{\Psi_\alpha}(p)=\lim_{n\to +\infty}\frac{\log^+\vert\vert\Psi_\alpha^n(p)\vert\vert}{q^n}=\lim_{n\to +\infty}\frac{\log^+\vert\vert\theta \circ \Psi_\alpha^n(p)\vert\vert}{q^n}=G_{\Phi_\a}^+\circ \theta(p)=G_{\phi_\a}^+ \circ h (p).
\end{equation}
Now if $p=(p_0,p_1,0)$, then we have seen that $G_{\Psi_\a}^+(p)=0$. Note that 
$h(p)=0_{\C^2}$ and $\theta(p)=0_{\C^3}$; therefore, 
$G_{\Phi_\a}^+\circ \theta(p)=G_{\phi_\a}^+ \circ h (p)=0$. We conclude that 
(\ref{defgpsi}) holds for any point $p \in \C^3$. 

The function $G_{\Psi_\a}^+$ is not $\equiv -\infty$, it is upper semicontinuous 
and satisfies the sub-mean value property (since $G_{\Phi_\a}^+$ does and 
$\theta\colon \C^3\to \C^3$ is holomorphic); in other terms, $G_{\Psi_\a}^+$ is plurisubharmonic. 
Moreover, we know that $G_{\phi_\a}^+$ is Hölder continuous, and $h$ is holomorphic,
hence $G_{\Psi_\a}^+$ is Hölder continuous as well. 
We have shown:

\begin{prop}\label{propgreen}
For any point $p \in \C^3$, the limit 
\[
\displaystyle\lim_{n\to +\infty}\frac{\log^+\vert\vert\Psi_\alpha^n(p)\vert\vert}{q^n}=:G^+_{\Psi_\alpha}(p)
\] 
exists; the function $G^+_{\Psi_\alpha}=G_{\Phi_\a}^+ \circ \theta=G_{\phi_\a}^+ \circ h$ is  plurisubharmonic, 
Hölder continuous, and satisfies $G_{\Psi_\a}^+\circ \Psi_\a = q \cdot G_{\Psi_\a}^+$. We can then 
define the positive current $T_{\Psi_\a}^+:=\mathrm{d d^c}G^+_{\Psi_\alpha}$. The maps 
$\restriction{\theta}{\mathcal{U}}$ and $\restriction{h}{\mathcal{U}}$ are submersions, and 
$\restriction{T_{\Psi_\a}^+}{\mathcal{U}}= (\restriction{\theta}{\mathcal{U}})^*(\restriction{T_{\Phi_\a}^+}{\mathcal{U}})=(\restriction{h}{\mathcal{U}})^*(\restriction{T_{\phi_\a}^+}{\mathcal{U}})$. We also have
$\Psi_\a^* (T_{\Psi_\a}^+)=q\cdot T_{\Psi_\a}^+$.
\end{prop}

\begin{rem}
We observe that contrary to the case of $\phi_\a$, the set $K_{\Psi_\a}^+$ of points 
whose forward orbit is bounded is stricly contained in $\big\{p\in\C^3\ \vert\ G_{\Psi_\a}^+(p)=0\big\}$; 
indeed, we have seen that the latter always contains $\{z_2=0\} \not\subset K_{\Psi_\a}^+$. 
\end{rem}

\section{Analysis of the dynamics of the automorphism $\Psi_\a$}

In this section, we further analyze the dynamics of the automorphism $\Psi_\a$, 
distinguishing between the value of $0<|\alpha|\leq 1$. In particular, we want 
to describe what happens outside the invariant hypersurface $\{z_2=0\}$, 
where we have seen that the dynamics corresponds to the one of 
a linear Anosov diffeomorphism.

We see that a transition occurs for $|\a| = \varphi^{(1-q)/d}$. Indeed, 
when $|\a| < \varphi^{(1-q)/d}$, we observe different behaviors in
the escape speed outside $\{z_2=0\}$ according to the choice
of the starting point $p=(p_0,p_1,p_2)$: Fibonacci, or bigger than 
$\eta^{q^n}$ for some $\eta > 1$. On the contrary, for $|\a| > \varphi^{(1-q)/d}$, 
we see that it is impossible to escape to infinity with Fibonacci speed, 
while the second case persists. 

Let us say a few words about the critical value $\varphi^{(1-q)/d}$ where the
 transition happens. We define the cocycle $A\colon\C^3 \to \mathrm{GL}_2(\C)$ by:
\[
A(z_0,z_1,z_2):=\begin{pmatrix}
1+z_0^{q-1}z_2^d & 1\\
1 & 0
\end{pmatrix},
\]
and if $M\in \mathfrak{M}_2(\C)$ and $v=(v_0,v_1) \in \C^2$, we set $M\cdot v:=vM^T$. 
Recall that for every $z_2 \in \C$, we consider $\psi_{z_2}=(z_0+z_1+z_0^q z_2^d,z_0)$. 
We remark that for every $p=(p_0,p_1,p_2) \in \C^3$, $\psi_{p_2}(p_0,p_1)=A(p)\cdot(p_0,p_1)$. 
As usual we denote $A_0(p):=\mathrm{Id}$ and for $n \geq 1$,
\[
A_n(p):=A(\Psi_\a^{n-1}(p)) \cdot A(\Psi_\a^{n-2}(p))\ \dots \  A(\Psi_\a(p)) \cdot A(p).
\]
In particular, for every $n \geq 0$, $\Psi_\a^n(p)=(A_n(p) \cdot (p_0,p_1),\a^n p_2)$; equivalently, 
$(P_\a^{(n)}(p),P_\a^{(n-1)}(p))=A_n(p) \cdot (p_0,p_1)$. Note that 
\[
A(\Psi_\a^n(p))=\begin{pmatrix}
1+(P_\a^{(n)}(p))^{q-1}\a^{nd} p_2^d & 1\\
1 & 0
\end{pmatrix}.
\]
\begin{itemize}
\smallskip
\item If $\lim\limits_{n \to +\infty} \big(P_\a^{(n)}(p)\big)^{q-1}\a^{nd} = 0$, then 
$\lim\limits_{n \to +\infty} A(\Psi_\a^n(p)) = \begin{pmatrix}
1 & 1\\
1 & 0
\end{pmatrix}$, whose largest eigenvalue is $\varphi$.
Then the growth will be exactly Fibonacci unless the initial point 
belongs to $W_{\Psi_\a}^s(0_{\C^3})$. But then 
\[
\big(\varphi^{q-1} |\alpha|^d\big)^n=O\big(|P_\a^{(n)}(p)|^{q-1}|\a|^{nd}\big)=o(1),
\] 
and necessarily $|\a| < \varphi^{(1-q)/d}$.
\smallskip
\item If $|\a| > \varphi^{(1-q)/d}$, Fibonacci growth is impossible; indeed if 
$|P_\a^{(n)}(p)| \geq C \varphi^n$ with $C >0$, 
then 
\[
|P_\a^{(n)}(p)|^{q-1}|\a|^{nd} \geq C^{q-1} \big(\varphi^{q-1} |\alpha|^d\big)^n 
\]
but here $\eta:=\varphi^{q-1} |\alpha|^d > 1$ so that 
$|P_\alpha^{(n+1)}(p)| \gtrsim C^{q-1} \eta^{n} |P_\alpha^{(n)}(p)|$ and the growth is much more in fact.
\end{itemize}

We will also see that this transition reflects an analogous 
change in the dynamics of the Hénon automorphism $\phi_\a$: for $|\alpha|<\varphi^{(1-q)/d}$, 
the point $0_{\C^2}$ is a sink of $\phi_\a$, while for  $|\alpha|>\varphi^{(1-q)/d}$, the point 
$0_{\C^2}$ becomes a saddle fixed point.\\

The following general lemma will be useful in the analysis that follows. 
\begin{lem}\label{proplessfib}
Assume that $0<\vert\alpha\vert\leq1$, and that
$p=(p_0,p_1,p_2)$
satisfies: 
\[
|P_\a^{(n)}(p)|=O([(1-\varepsilon)\varphi]^n).
\]
Then with our previous notations, $p \in \mathcal{Z}$. 
\end{lem}

\begin{proof}
This follows again from Lemma \ref{lemmarecphi}. Indeed for every $n \geq 0$,
\[
\left|g_n(p)\right| \leq \big(|P_\a^{(n+1)}(p)|+
\varphi^{-1}|P_\a^{(n)}(p)|\big)\varphi^{-n}= O\big((1-\varepsilon)^n\big)
\]
hence $g(p)=\lim\limits_{n \to +\infty} g_n(p)=0$.
\end{proof}

\subsection{Points escaping to infinity with maximal speed}

The results of this subsection hold for any $0<|\alpha| \leq 1$. 
We start by exhibiting an explicit non-empty open set of points 
escaping to infinity very fast; then we state some facts concerning the set of 
points going to infinity with maximal speed, and show how they can be derived 
from the properties of the Green function $G_{\Psi_\a}^+$.

Set 
$\gamma:=\frac{\ln\vert\alpha\vert}{\ln(\varphi)}$. We choose $M\geq 0$ sufficiently large so that 
$M(q-1)+d \gamma > 0$ (this is possible since by hypothesis, $q-1 >0$). 
  
\begin{prop}\label{explss}
We define the open set
\[
\Omega:=\big\{p=(p_0,p_1,p_2)\in\mathbb{C}^3\,\vert\,\  \vert p_0\vert > \vert p_1\vert>0\ \text{and}\
\vert p_1\vert^{q-1} \vert p_2\vert^d > 2+\varphi^M\big\}.
\]
Then for any point $p\in \Omega$, we have $G_{\Psi_\a}^+(p)>0$;
moreover the sequence $(|P_\a^{(n)}(p)|)_n$ is increasing.
\end{prop}
  
The proof splits in two lemmas that we are going to detail now. 

\begin{lem}\label{lem:superpol}
For any point $p$ in $\Omega$ the escape speed is superpolynomial: for any $n \geq -1$,
\begin{equation}\label{fibonacci1}
\vert P_\alpha^{(n)}(p)\vert \geq \vert p_1\vert \varphi^{Mn}.
\end{equation} 
Moreover the sequence $(|P_\a^{(n)}(p)|)_n$ is increasing.
\end{lem}

\begin{proof}
The proof is by induction on $n \geq -1$. Let $p \in \Omega$; we will show that 
$(|P_\a^{(n)}(p)|)_n$ is increasing and that (\ref{fibonacci1}) holds.
It follows from our assumptions that
\smallskip
\begin{itemize}
\item $\vert P_\alpha^{(-1)}(p)\vert = |p_1| \geq |p_1| \varphi^{-M}$.
\smallskip
\item $\vert P_\alpha^{(0)}(p)\vert=\vert p_0\vert \geq |p_1|$.
\smallskip
\item Take $n \geq 0$ and assume that 
 $\vert P_\alpha^{(n)}(p)\vert \geq |p_1|\varphi^{Mn}$ and $|P_\a^{(n)}(p)| \geq |P_\a^{(n-1)}(p)|$. 
We estimate:
\begin{align*}
|P_\a^{(n+1)}(p)| &\geq |P_\a^{(n)}(p)|(|P_\a^{(n)}(p)|^{q-1} |\a^n p_2|^d -2)\\
&\geq |P_\a^{(n)}(p)| (|p_1|^{q-1} |p_2|^d \varphi^{(M(q-1)+d \gamma)n} -2)\\
&\geq |p_1| \varphi^{Mn} (|p_1|^{q-1} |p_2|^d -2)\\
&\geq |p_1| \varphi^{M(n+1)}
\end{align*}
because $M(q-1)+d \gamma > 0$ and $|p_1|^{q-1} |p_2|^d -2 > \varphi^M$. 
Since $|p_1|^{q-1} |p_2|^d -2 \geq 1$, the previous inequalities also show that 
$|P_\a^{(n+1)}(p)| \geq |P_\a^{(n)}(p)|$, which concludes the induction.
\end{itemize}
\end{proof}

\begin{lem}\label{lemmaspeedexplosion}
Recall that $\gamma:=\frac{\ln\vert\alpha\vert}{\ln(\varphi)}$ and that $M\geq 0$ is chosen such that $M(q-1)+d \gamma > 0$.
Take $p \in \C^3$ such that the sequence $(|P_\a^{(n)}(p)|)_{n\geq 0}$
is increasing, and assume that there exists $n_0 \geq 1$ such that for every $n \geq n_0$, the following 
inequality holds\footnote{In particular, this is 
satisfied for points $p \in \Omega$ as we have seen
in Lemma \ref{lem:superpol}.}:
\[
|P_\a^{(n)}(p)| \geq \varphi^{Mn}.
\]
Then the escape speed is much bigger in fact: there exist $n_1 \geq n_0$ and $\eta >1$ such that 
for every $n \geq n_1$,
\[
|P_\a^{(n)}(p)| \geq \eta^{q^n}.
\]
In terms of the Green function introduced above, we then get $G_{\Psi_\a}^+(p)>0$.
\end{lem}

\begin{proof}
Since $(|P_\a^{(n)}(p)|)_{n\geq 0}$ is increasing, we have for every $n \geq 0$, 
\begin{equation}\label{inegpa}
|P_\a^{(n)}(p)|^q (|\a|^n |p_2|)^d = |P_\a^{(n+1)}(p)-P_\a^{(n)}(p)-P_\a^{(n-1)}(p)| \leq 3|P_\a^{(n+1)}(p)|.
\end{equation}
Set $x_n:=\ln |P_\a^{(n)}(p)|$. From our hypotheses, we know that for every $n \geq n_0$,
$x_n \geq Mn \ln \varphi$. Since $M(q-1)+d \gamma > 0$, we can take $\varepsilon >0$ small such that we 
still have $M(q-1-\varepsilon)+d \gamma > 0$. Let $n_0' \geq n_0$ be chosen such that for $n \geq n_0'$, 
$n(M(q-1-\varepsilon)+d \gamma)\ln \varphi +d \ln |p_2| - \ln 3 \geq 0$. 
Thanks to (\ref{inegpa}), we get: for every $n \geq n_0'$,
\begin{align*}
x_{n+1} &\geq q x_n + n d \gamma \ln\varphi + d \ln |p_2| - \ln 3\\
&\geq (1+\varepsilon) x_n + \left(n(M(q-1-\varepsilon)+d \gamma)\ln \varphi +d \ln |p_2| - \ln 3 \right)\\
&\geq  (1+\varepsilon) x_n.
\end{align*}
We then obtain: for every $n \geq n_0'$, 
\[
x_n \geq (1+\varepsilon)^{n-n_0'} x_{n_0'} \geq (1+\varepsilon)^{n-n_0'} M n_0' \ln \varphi.
\]
As a result there exists $n_1 \geq n_0'$ such that for $n\geq n_1$, 
\[
\frac{x_n}{n^2} \geq \frac{(q/2)^{n-n_0'} M n_0' \ln \varphi}{n^2} \geq -( n d \gamma \ln\varphi +d \ln |p_2| - \ln 3).
\] 
We can then refine the previous inequalities: for $n \geq n_1$,
\[
x_{n+1}\geq q x_n + n d \gamma \ln\varphi +d \ln |p_2| - \ln 3\geq q \left(1-\frac{1}{n^2}\right) x_n.
\]
Let $C:=\prod\limits_{n \geq n_1} \left(1-\frac{1}{n^2}\right) x_{n_1} q^{-n_1}>0$; 
for every $n \geq n_1$, we have $x_n \geq C q^n$, hence $|P_\a^{(n)}(p)| \geq\eta^{q^n}$, where 
$\eta:=e^{C} > 1$.
\end{proof}

\begin{rem}
We have seen that the automorphism $\phi_\a$ 
possesses an attractor at infinity 
$X(\phi_\a)=(1:0:0)$ whose basin is biholomorphic to $\C^2$. 
Then there exists $C=C(\alpha)>0$ such that the forward orbit of 
any point $\widetilde p=(p_0,p_1) \in \C^2$ 
such that $|p_0| \gg |p_1|$ and $\|\widetilde p\| \geq C$ is attracted by $X(\phi_\a)$; 
in particular, $\widetilde p \not\in 
K_{\phi_\a}^+$, and thus, $G_{\phi_\a}^+(\widetilde p) >0$. 
If $p=(p_0,p_1,p_2)\in \C^3$ satisfies $|p_0| \gg |p_1|$ and 
$\|(p_0,p_1)\|\cdot |p_2|^l \geq~C$, 
we see that $\|h(p)\|  \geq C$, hence $G_{\phi_\a}^+(h(p)) >0$, 
and $G_{\Psi_\a}^+(p) >0$ as well. The definition of the set $\Omega$ in Proposition \ref{explss} is coherent with this observation.
\end{rem}

\begin{prop}
Set $\widetilde l := 2\max(l,1)$. We have 
\[
1 \leq \limsup\limits_{\|p\| \to +\infty} \frac{G_{\Psi_\a}^+(p)}{\log \|p\|}\leq \widetilde l.
\]
The set $\mathcal{E}:=\big\{p \in \C^3\ \vert \ G_{\Psi_\a}^+(p) >0\big\}$ of points escaping 
to infinity with maximal speed is open, connected and of infinite measure on any complex line 
where $G_{\Psi_\a}^+$ is not identically zero. In particular, the set 
\[
\big\{p \in \C^3\ \vert \ \lim\limits_{n \to +\infty}\|\Psi_\a^n(p)\| = +\infty\big\} 
\]
of points whose forward orbit goes to infinity is of infinite measure. 
\end{prop}

\begin{proof}
The openness of $\mathcal{E}$ follows directly from  
the continuity of $G_{\Psi_\a}^+$, shown in Proposition \ref{propgreen}.

The proof of the fact that $\mathcal{E}$ has infinite measure follows 
arguments given by Guedj-Sibony \cite{GuedjSibony}. Since~$\phi_\a$ is 
algebraically stable, we know from Proposition 1.3 in \cite{GuedjSibony} that 
$\limsup\limits_{\|\widetilde p\| \to +\infty} \frac{G_{\phi_\a}^+(\widetilde p)}
{\log \|\widetilde p\|}=1$. 
Therefore 
\[
\limsup\limits_{\|p\| \to +\infty} \frac{G_{\Psi_\a}^+(p)}{\log \|p\|}=\limsup\limits_{\|p=(p_0,p_1,p_2)\| \to +\infty} \frac{G_{\phi_\a}^+ \circ h(p)}{\log \|h(p)\|}\times \frac{l\log|p_2|+\log\|(p_0,p_1)\|}{\log \|p\|}\leq \widetilde l \limsup\limits_{\|\widetilde p\| \to +\infty} \frac{G_{\phi_\a}^+(\widetilde p)}
{\log \|\widetilde p\|} = \widetilde l. 
\] 
For the other inequality, we remark that
\[
\limsup\limits_{\|p\| \to +\infty} \frac{G_{\Psi_\a}^+(p)}{\log \|p\|}\geq \limsup\limits_{\|p=(p_0,p_1,1)\| \to +\infty} \frac{G_{\phi_\a}^+ \circ h(p)}{\log \|h(p)\|}\times \frac{ \log \|(p_0,p_1)\|}{\log \|(p_0,p_1,1)\|}=  \limsup\limits_{\|\widetilde p\| \to +\infty} \frac{G_{\phi_\a}^+(\widetilde p)}{\log \|\widetilde p\|} =1.
\]
Assume that $p \in \C^3$ satisfies $G_{\Psi_\a}^+(p)>0$, and for some $v \neq 0_{\C^3}$, consider 
the line $L:=\big\{p+t v\ \vert \ t \in \C\big\}$. Denote by $m(r)$ the Lebesgue measure 
of the set $\big\{e^{i x},\ x \in \mathbb{R}\ \vert\ G_{\Psi_\a}^+(p+r e^{ix} v) >0\big\}$. 
From what precedes, we know that there exists $C> 0$ such that for every $r \geq 0$, 
\[
G_{\Psi_\a}^+(p+r e^{ix} v) \leq \widetilde l \log^+(r)+ C. 
\]
By the sub-mean value property, 
\[
0<G_{\Psi_\a}^+(p)\leq \frac{1}{2 \pi} \int_0^{2 \pi} G_{\Psi_\a}^+(p+r e^{ix} v)  dx \leq 
\frac{1}{2 \pi} (\widetilde l \log^+(r) + C) m(r).
\]
Therefore, $m(r) \geq \frac{2\pi G_{\Psi_\a}^+(p)}{\widetilde l \log^+(r) + C}$, and 
integrating over $r$, we get that the set of points $p$ in $L$ such that 
$G_{\Psi_\a}^+(p) >0$ has infinite measure. The proof of connectivity 
is also based on the slow growth of $G_{\Psi_\a}^+$ and follows from 
similar arguments
(see \cite{Heins}). 
\end{proof}

\subsection{General remarks when $0<|\alpha|<1$}\label{subsgeneralrem}
 
In this case, $0_{\C^3}$ is a hyperbolic fixed point of $\Psi_\a$ of saddle type, 
and Corollary \ref{corvarstable} tells us that the set $K_{\Psi_\a}^+$ 
of points with bounded forward orbit is exactly the stable manifold 
$W_{\Psi_\a}^s(0_{\C^3})$. The positive Julia set $J_{\Psi_\a}^+$ thus corresponds 
to $\partial K_{\Psi_\a}^+=\overline{W_{\Psi_\a}^s(0_{\C^3})}$. Let $p=(p_0,p_1,p_2) \in \C^3$. For $n \geq 0$, 
\begin{equation}\label{eqconjugbiss}
\theta \circ \Psi_\a^n(p)= \big(P_\a^{(n)}(p) (\a^{n}p_2)^l,P_\a^{(n-1)}(p)(\a^{n}p_2)^l,\a^n p_2\big)=\big(\phi_\a^n \circ h (p),\alpha^n p_2\big).
\end{equation}
From (\ref{eqconjugbiss}), we get that $h(W_{\Psi_\a}^s(0_{\C^3}))=W_{\phi_\a}^s(0_{\C^2})=K_{\phi_\a}^+$.\footnote{Indeed, if $p \in W_{\Psi_\a}^s(0_{\C^3})$, then $h(p) \in W_{\phi_\a}^s(0_{\C^2})$; conversely, if $(p_0,p_1) \in K_{\phi_\a}^+$, then $(p_0,p_1)=h(p_0,p_1,1)$ and $(p_0,p_1,1) \in K_{\Psi_\a}^+=W_{\Psi_\a}^s(0_{\C^3})$.} We deduce that 
$J_{\phi_\a}^+=\partial W_{\phi_\a}^s(0_{\C^2})$.
Besides, we know that $K_{\phi_\a}^+=\big\{\widetilde p\in \C^2\ \vert\ G_{\phi_\a}^+(\widetilde p)=0\big\}$, and this set is closed by continuity of $G_{\phi_\a}^+$.
In particular, $W_{\phi_\a}^s(0_{\C^2})$ is closed. For any $\widetilde p \in \C^2$, there are two possible behaviors: either $p \in W_{\phi_\a}^s(0_{\C^2})$ and then 
its forward iterates 
converge to $0_{\C^2}$ exponentially fast, or they go to infinity with maximal speed.

\subsection{Analysis of the dynamics in the case where $0<\vert\alpha\vert<\varphi^{(1-q)/d}$}
We show that under this assumption, we can construct a set of points with 
non-empty interior
for which the escape speed is much smaller, in fact Fibonacci.

From our hypothesis on $\alpha$, we can take 
$\varepsilon>0$ small enough so that 
$\eta:=((1+\varepsilon)\varphi)^q |\a|^d < \varphi$.

\begin{prop}\label{Fibonaccigrowth}
Assume $0<\vert\alpha\vert<\varphi^{(1-q)/d}$. We consider the following open neighborhood of the hypersurface $\{z_2=0\}$:
\[
\Omega':=\big\{p=(p_0,p_1,p_2)\in\mathbb{C}^3\,\vert\, (|p_0|+|p_1|)^{q-1} |p_2|^d < \varphi \varepsilon \big\}.
\] 
If $p\in \Omega'$, there are two possible behaviors:
\smallskip
\begin{itemize}
\item either $p$ belongs to the stable manifold $W_{\Psi_\a}^s(0_{\C^3})$ and then 
its forward iterates converge to $0_{\C^3}$ exponentially fast;
\smallskip
\item or $p$ goes to infinity with Fibonacci speed: $\big(P_{\a}^{(n)}(p) \varphi^{-n}\big)_{n\geq 0}$ converges and we have
\[
\lim_{n \to +\infty}P_{\a}^{(n)}(p) \varphi^{-n} \in \C^*.
\] 
\end{itemize}
\end{prop}

\begin{rem}
The last result tells us that if we start close enough to $\{z_2=0\} \subset \Omega'$, 
the dynamics is similar to the one we observe on restriction to 
this invariant hypersurface: 
either the starting point belongs to $W_{\Psi_\a}^s(0_{\C^3})$ and in this case its forward orbit 
converges to $0_{\C^3}$ with exponential speed, or the iterates escape to infinity with 
 speed exactly Fibonacci. We remark that when $|\alpha|$ is small, we can choose $\varepsilon >0$ reasonably large, 
so that the set $\Omega'$ becomes 
larger and larger. This is coherent with the fact that the smaller $|\alpha|$ is, the faster
we converge to the hypersurface $\{z_2=0\}$.
\end{rem}

We start by showing that the speed cannot be more than Fibonacci. 

\begin{lem}\label{proplessfib1}
Any point $p \in \Omega'$ grows at most with Fibonacci speed, that is,
there exists $C =C(p_0,p_1)>0$ such that for any $n \geq 0$,
\[
|P_\a^{(n)}(p)| \leq C \varphi^n.
\]
\end{lem}

\begin{proof}
Let $\widetilde C=\widetilde C(p_0,p_1):=|p_0|+|p_1|$ and $\varepsilon >0$ be chosen as explained above. We first show that for every $n \geq 0$,
\[
|P_\a^{(n)}(p)| \leq \widetilde C ((1+\varepsilon)\varphi)^n.
\]
The result is clearly true for $n=0$, and for $n=1$ we have
\[
|P_\alpha^{(1)}(p)|\leq |p_0| + |p_1| + |p_0|^q |p_2|^d\leq \widetilde C (1 +\widetilde C^{q-1} |p_2|^d) \leq \widetilde C (1+\varepsilon) \varphi.
\] 
Suppose that it holds for $n-1$ and $n$, that is
\[
 |P_\a^{(n-1)}(p)| \leq  \widetilde C((1+\varepsilon)\varphi)^{n-1},\quad \quad |P_\a^{(n)}(p)| \leq \widetilde C((1+\varepsilon)\varphi)^n.
\]
We then have 
\begin{align*}
|P_\a^{(n+1)}(p)|&\leq |P_\a^{(n)}(p)| + |P_\a^{(n-1)}(p)| +|(P_\a^{(n)}(p))^q (\alpha^{n}p_2)^d|\\
&\leq \widetilde C\big((1+\varepsilon)\varphi\big)^n +\widetilde C \big((1+\varepsilon)\varphi\big)^{n-1} 
+ \widetilde C^q |p_2|^d \big(((1+\varepsilon)\varphi)^q |\a|^d \big)^n\\
&\leq \widetilde C (1+\varepsilon)^n \varphi^{n+1} + \widetilde C \varphi \varepsilon \eta^n\\
&\leq \widetilde C (1+\varepsilon)^n \varphi^{n+1} + \widetilde C \varepsilon (1+\varepsilon)^n \varphi^{n+1}\\
&= \widetilde C \big((1+\varepsilon) \varphi\big)^{n+1},
\end{align*}
which concludes the induction. 

Using this fact, we obtain a good control on the non-linear term:
for any $n \geq 0$,
\[
|(P_\a^{(n)}(p))^q (\a^{n}p_2)^d| \leq \widetilde C^q |p_2|^d \big(((1+\varepsilon)\varphi)^q |\a|^d \big)^n= C_0 \eta^n,
\]
where $C_0=C_0(p_0,p_1):=\widetilde C \varphi \varepsilon$.
For every $n \geq 0$, we have
\begin{align*}
|P_\a^{(n+1)}(p)| &\leq |P_\a^{(n)}(p)| + |P_\a^{(n-1)}(p)| +|(P_\a^{(n)}(p))^q (\a^{n}p_2)^d|\\
&\leq |P_\a^{(n)}(p)| + |P_\a^{(n-1)}(p)| + C_0 \eta^n.
\end{align*}
Thanks to the same trick as in the proof of Lemma \ref{lemmarecphi}, we obtain:
\begin{equation*}
|P_\a^{(n+1)}(p)| + (\varphi -1) |P_\a^{(n)}(p)| \leq \varphi^n \left(\varphi |p_0| + |p_1| + 
C_0 \sum\limits_{j=0}^{n} \left(\frac{\eta}{\varphi}\right)^j\right).
\end{equation*}
Since $\eta < \varphi$, we can set 
$C=C(p_0,p_1):=\varphi^{-1}\left(\varphi|p_0|+|p_1|+C_0\sum\limits_{j=0}^{+\infty}\left(\frac{\eta}{\varphi}\right)^j\right)$. 
We then get: for every $ n\geq 0$,
\[
|P_\a^{(n)}(p)| \leq C \varphi^n.
\]
\end{proof}

The proof of Proposition \ref{Fibonaccigrowth} 
is the combination of Lemma \ref{proplessfib1} and of the next 
result.

\begin{lem}\label{lemutile}
Assume $0<\vert\alpha\vert<\varphi^{(1-q)/d}$ and take 
$p \in \C^3\smallsetminus W_{\Psi_\a}^s(0_{\C^3})$ with speed less than Fibonacci, i.e.,
 there exists $C >0$
 such that for every $n \geq 0$, $|P_\a^{(n)}(p)| \leq C \varphi^n$. Then $p$ goes to
 infinity with speed exactly Fibonacci: $\big(P_{\a}^{(n)}(p) \varphi^{-n}\big)_{n\geq 0}$ 
converges and we have
\[
\lim_{n \to +\infty}P_{\a}^{(n)}(p) \varphi^{-n} \in \C^*.
\] 
\end{lem}

\begin{proof}
Take $p \in \C^3 \smallsetminus W_{\Psi_\a}^s(0_{\C^3})$ such that for every 
$n \geq 0$, $|P_\a^{(n)}(p)| \leq C \varphi^n$, $C >0$. 
We first show that $p$ goes to infinity with speed 
at least Fibonacci too, i.e., there exists $C'=C'(p_0,p_1)>0$ such 
that for every $n \geq 0$, $|P_\a^{(n)}(p)| 
\geq C' \varphi^n$. 
Recall that if $(z_0,z_1,z_2) \in \C^3$, we denote
\[
A(z_0,z_1,z_2):=\begin{pmatrix}
1+z_0^{q-1}z_2^d & 1\\
1 & 0
\end{pmatrix}, 
\]
and that for $n \geq 0$, $\big(P_\a^{(n)}(p),P_\a^{(n-1)}(p)\big)=A_n(p) \cdot (p_0,p_1)$, where
\[
A_n(p):=A(\Psi_\a^{n-1}(p)) \cdot A(\Psi_\a^{n-2}(p))\ \dots\ A(\Psi_\a(p)) \cdot A(p).
\]
Note that for every $j \geq 0$,
\begin{equation}\label{matrixcocycle}
A(\Psi_\a^j(p))=\begin{pmatrix}
1+(P_\a^{(j)}(p))^{q-1}\a^{jd} p_2^d & 1\\
1 & 0
\end{pmatrix}.
\end{equation}
Since $|P_\a^{(j)}(p)| 
\leq C \varphi^j$, we see that 
\begin{equation}\label{estiiterates}
|(P_\a^{(j)}(p))^{q-1}\a^{jd} p_2^d| \leq C^{q-1}|p_2|^d
(\varphi^{q-1}|\a|^d)^j \leq \nu \eta^j,
\end{equation}
where $\eta:=\varphi^{q-1}|\a|^d<1$ and $\nu:=C^{q-1}|p_2|^d \geq 0$. Set
$M_0:=\begin{pmatrix}
1 & 1\\
1 & 0
\end{pmatrix}$. If $\varepsilon >0$, let us consider 
$B_{\varepsilon}(M_0):=\big\{M\in \mathfrak{M}_2(\C)\ \vert \ \|M-M_0\|_\infty< \varepsilon\big\}$. 
We see from (\ref{matrixcocycle}) and (\ref{estiiterates}) that for every $j \geq 0$,
 $A(\Psi_\a^{j}(p))\in B_{\nu\eta^j}(M_0)$. For $\varepsilon >0$ small, 
 every matrix $M \in B_{\varepsilon}(M_0)$ is hyperbolic with eigenvalues close to 
$\varphi$ and $\varphi'$; moreover we can choose a family of cones
$(\mathscr{C}_{\varepsilon})_{\varepsilon>0}$ around $\Delta_\varphi$ satisfying the following: 
there exist $C_0,C_1 >0$ such that for each $M \in B_{\varepsilon}(M_0)$,
every vector $v \in \mathscr{C}_\varepsilon$
will be expanded by a factor close to $\varphi$: 
\[
(1-C_0 \varepsilon)\varphi \|v\| \leq \|M\cdot v\| \leq (1+C_1 \varepsilon)\varphi \|v\|.
\]
Since $p \not\in W_{\Psi_\a}^s(0_{\C^3})$, the iterates of $(p_0,p_1)$ are expanded 
and accumulate on the unstable 
space $\Delta_\varphi=\big\{(\varphi z,z)\ \vert\ z \in \C\big\}$ of $M_0$. In fact the angle 
$\angle(\Delta_\varphi, A_j(p) \cdot (p_0,p_1))$
decreases exponentially fast, and
we can assume that for some $n_0 \geq 0$ and for every $j \geq n_0$,
$A(\Psi_\a^{j}(p))\in B_{\nu\eta^j}(M_0)$ maps 
$\mathscr{C}_{\nu \eta^j}$ to $\mathscr{C}_{\nu\eta^{j+1}}$.
We deduce that for every $n \geq n_0$, 
\[
\prod\limits_{j=n_0}^{n-1}(1-C_0 \nu \eta^j) \|A_{n_0}(p) \cdot (p_0,p_1)\| \leq \frac{\|A_n(p) 
\cdot (p_0,p_1)\|}{\varphi^{n-n_0}} \leq \prod_{j=n_0}^{n-1}(1+C_1 \nu \eta^j) \|A_{n_0}(p) \cdot (p_0,p_1)\|.
\]
Let $C':=\varphi^{n_0}\displaystyle\prod\limits_{j=n_0}^{+\infty}(1-C_0 \nu \eta^j) \|A_{n_0}(p) \cdot (p_0,p_1)\|>0$. 
We have thus obtained: for every $n \geq 0$,  
\[
\|(P_\a^{(n)}(p),P_\a^{(n-1)}(p))\| \geq C'\varphi^n.
\]

Now, we note that $p$ belongs to $\mathcal{D}$, the domain of definition of 
the series $g$ introduced earlier. Indeed, for any $j \geq 0$,
$|P_\a^{(j)}(p)|^q \varphi^{-j} |\alpha|^{jd} \leq C^q (\varphi^{q-1}|\a|^d)^j$ and 
$\varphi^{q-1}|\a|^d < 1$. 
We have shown that the sequence 
$(P_{\a}^{(n)}(p),P_{\a}^{(n-1)}(p))_n$ accumulates on the unstable 
direction $\big\{(\varphi z,z)\ \vert\ z \in \C\big\}$ of $M_0$; therefore, we get
\begin{equation}\label{groww1}
\lim\limits_{n \to +\infty} \frac{P_{\a}^{(n)}(p)}{P_{\a}^{(n-1)}(p)}=\varphi.
\end{equation}
Recall that for $n \geq 0$, $g_n(p):=(P_{\a}^{(n+1)}(p)+\varphi^{-1}P_{\a}^{(n)}(p))\varphi^{-n}$. 
We have seen that $p \in \mathcal{D}$, and then, $(g_n(p))_n$ converges. From (\ref{groww1})
 we deduce that for every $n \geq 0$,
\[
g_n(p)=\left(\frac{P_{\a}^{(n+1)}(p)}{P_{\a}^{(n)}(p)}+\varphi^{-1}\right)P_{\a}^{(n)}(p) \varphi^{-n}\sim \left(\varphi+\varphi^{-1}\right)P_{\a}^{(n)}(p) \varphi^{-n}.
\]
This implies that $(P_{\a}^{(n)}(p) \varphi^{-n})_{n \geq 0}$ converges. 
But we also know from what precedes that $\lim\limits_{n\to +\infty}  |P_{\a}^{(n)}(p) \varphi^{-n}|>0$, 
so that $\lim\limits_{n\to +\infty}  P_{\a}^{(n)}(p) \varphi^{-n}\in \C^*$, which concludes.
\end{proof}

Recall that we take $\varepsilon>0$ small enough so that 
$((1+\varepsilon)\varphi)^q |\a|^d < \varphi$, and that 
$\Omega':=\big\{(p_0,p_1,p_2)\in\mathbb{C}^3\,\vert\, (|p_0|+|p_1|)^{q-1} |p_2|^d < \varphi \varepsilon \big\}$.
We deduce from Proposition \ref{Fibonaccigrowth} that
\begin{equation}\label{eqpot}
\restriction{G_{\Psi_\a}^+}{\Omega'}=0.
\end{equation}
Indeed, the forward iterates of any point 
in $\Omega'$ grow at most with Fibonacci speed as we have seen. 
In particular, the set $\big\{p\in \C^3 \ \vert \ G_{\Psi_\a^+}(p)=0\big\}$ has non-empty interior. 
Set $\delta:=(\varphi \varepsilon)^{1/(q-1)}$ and define the open ball 
\[
B_\delta:=\big\{\widetilde{p}=(p_0,p_1) \in \C^2\ \vert \ \| \widetilde{p}\|_1=|p_0|+|p_1|< \delta\big\}.
\]
Recall that $l=\frac{d}{q-1}$ and that $h\colon(z_0,z_1,z_2)\mapsto (z_0 z_2^l,z_1 z_2^l)$. Remark 
that $h(\Omega') \subset B_\delta$. Indeed, if $p=(p_0,p_1,p_2) \in \Omega'$, then
\[
\|h(p)\|_1=|p_0p_2^l| + |p_1p_2^l| = \left((|p_0|+|p_1|)^{q-1} |p_2|^d\right)^{1/(q-1)} < (\varphi \varepsilon)^{1/(q-1)}=\delta.
\]
Conversely, if $(p_0,p_1) \in B_\delta$, 
then $(p_0,p_1)=h(p_0,p_1,1)$ with  $(p_0,p_1,1)\in \Omega'$, so that $h(\Omega') =B_\delta$ in fact. 
Since $G_{\Psi_\a}^+=G_{\phi_\a}^+\circ h$, 
we deduce 
from (\ref{eqpot}) that 
\[
\restriction{G_{\phi_\a}^+}{B_\delta}=0.
\]
But as we have seen, $K_{\phi_\a}^+=\big\{(p_0,p_1) \in \C^2\ \vert\ G_{\phi_a}^+(p_0,p_1)=0\big\}$. 
We conclude that for $\vert\alpha\vert<\varphi^{(1-q)/d}$, any point $(p_0,p_1) \in B_\delta$ has 
bounded forward orbit under $\phi_\a$. Actually we can say more. Recall that 
$\phi_\a=\alpha^l (z_0+z_1+z_0^q,z_0)$. We see that $0_{\C^2}$ is a sink of 
$\phi_\a$; indeed, the largest eigenvalue of the Jacobian is $\alpha^l \varphi$, 
which is stricly smaller than $1$ from the assumption 
we made on $\alpha$.
For any $p \in \C^3$ and $n \geq 0$,
\begin{equation}\label{eqconjj}
\theta \circ \Psi_\a^n(p)= \big(P_\a^{(n)}(p) (\a^{n}p_2)^l,P_\a^{(n-1)}(p)(\a^{n}p_2)^l,\a^n p_2\big)=\big(\phi_\a^n \circ h (p),\alpha^n p_2\big).
\end{equation}
If $p \in \Omega'$, we know from Proposition \ref{Fibonaccigrowth} 
that there exists $C >0$ such that for any $n \geq 0$, $|P_\a^{(n)}(p)| \leq C \varphi^n$. 
But then, $|P_\a^{(n)}(p) (\alpha^{n}p_2)^l| \leq C |p_2|^l (\varphi^{q-1} |\alpha|^d)^{n/(q-1)}$, 
and $\varphi^{q-1} |\alpha|^d < 1$, so we deduce from (\ref{eqconjj}) and the equality 
$B_\delta=h(\Omega')$ that any point in $B_\delta$ 
goes to $0_{\C^2}$ by forward iteration of $\phi_\a$; equivalently, 
the basin of attraction $W_{\phi_\a}^s(0_{\C^2})$ 
of the sink $0_{\C^2}$ contains the ball
$B_\delta$. Recall also that it is a general fact that for a sink $p$ of $\phi_\a$, 
$W_{\phi_\a}^s(p)$ is biholomorphic to $\C^2$.\\
 
In the following, we will see how the previous results enable us to give a description of the 
dynamics of $\Psi_\a$ in the case where $0<\vert\alpha\vert<\varphi^{(1-q)/d}$. 
Let $p \in \C^3$. We have shown previously that
 $K_{\Psi_\a}^+=W_{\Psi_\a}^s(0_{\C^3})$, so 
assume that the forward orbit of $p$ under $\Psi_\a$ is not bounded. There are two possibilities:
\begin{itemize}
\smallskip
\item either $G_{\Psi_\a}^+(p)>0$ and the iterates of $p$ go to infinity with maximal speed; in 
this case, we also have $G_{\phi_\a}^+(h(p))=G_{\Psi_\a}^+(p)>0$;
\smallskip
\item or $G_{\Psi_\a}^+(p)=0$ and $G_{\phi_\a}^+(h(p))=G_{\Psi_\a}^+(p)=0$ too. But then we know from 
the general properties of $\phi_\a$ that $h(p) \in K_{\phi_\a}^+$. From 
 \S \ref{subsgeneralrem}, we also have
 $K_{\phi_\a}^+=W_{\phi_\a}^s(0_{\C^2})$.  
It follows from 
(\ref{eqconjj}) that $\lim\limits_{n \to +\infty} |P_\a^{(n)}(p) (\a^{n}p_2)^l |$ exists and vanishes. 
Since $l=d/(q-1)$, we deduce that $\lim\limits_{n \to +\infty} |P_\a^{(n)}(p)|^{q-1} |\a^{n}p_2|^{d} =0$, 
hence $\lim\limits_{n \to +\infty} A(\Psi_\a^n(p)) = M_0=\begin{pmatrix}
1 & 1\\
1 & 0
\end{pmatrix}$
where $A$ is the cocycle introduced earlier. Reasoning as before, and since by assumption 
$p \not \in W_{\Psi_\a}^s(0_{\C^3})$, we conclude accordingly that $p$ escapes to infinity with 
Fibonacci speed. 
\end{itemize}

We have thus shown:

\begin{prop}\label{descrdd}
When $0<\vert\alpha\vert<\varphi^{(1-q)/d}$, the point $0_{\C^3}$ is a saddle fixed point of 
$\Psi_\a$ of index $2$, and $J_{\Psi_\a}^+:=\partial K_{\Psi_\a}^+=\overline{ W_{\Psi_\a}^s(0_{\C^3})}$. 
Moreover, 
$\big\{p\in \C^3\ \vert\ G_{\Psi_\a}^+(p)=0\big\}=h^{-1}(W_{\phi_\a}^s(0_{\C^2}))=\Omega'' \sqcup W_{\Psi_\a}^s(0_{\C^3})$, 
where $\Omega''$ has non-empty interior (it contains the set $\Omega'\smallsetminus W_{\Psi_\a}^s(0_{\C^3})$), 
and the forward orbit of points that belong to it goes to infinity with Fibonacci speed. 
Moreover, $W_{\phi_\a}^s(0_{\C^2})$ 
is biholomorphic to $\C^2$, and 
\[
K_{\phi_\a}^+=W_{\phi_\a}^s(0_{\C^2})=\big\{\widetilde p\in \C^2\ \vert\ G_{\phi_\a}^+(\widetilde p)=0\big\}.
\]
We summarize this as follows:
\begin{equation}
\begin{array}{rcl}
& h & \\
\{z_2=0\} & \to & \{0_{\C^2}\};\\
K_{\Psi_\a}^+=W_{\Psi_\a}^s(0_{\C^3}) & \to & K_{\phi_\a}^+=W_{\phi_\a}^s(0_{\C^2});\\
\Omega'' & \to & W_{\phi_\a}^s(0_{\C^2});\\
\big\{p\in \C^3\ \vert\ G_{\Psi_\a}^+(p)>0\big\} & \to & \big\{\widetilde p\in \C^2\ \vert\ G_{\phi_\a}^+(\widetilde p)>0\big\}.
\end{array}
\end{equation}
\end{prop}

Thanks to the last statement, 
we now give an alternative 
description of the stable manifold $W_{\Psi_\a}^s(0_{\C^3})$ in terms 
of the set $\mathcal{Z}$ of zeros of the series $g$ introduced previously.

\begin{prop}\label{lemstablemfd}
Assume $0<|\a| < \varphi^{(1-q)/d}$. Set 
$\mathcal{V}:=\big\{p\in \C^3\ \vert\ G_{\Psi_\a}^+(p)=0\big\}=\Omega'' \sqcup W_{\Psi_\a}^s(0_{\C^3})$. 
Then $\mathcal{V}$ coincides with  
the domain of definition $\mathcal{D}$ of the series 
$g$ introduced earlier. Moreover, we have the following parametrization of the stable manifold:
\[
W_{\Psi_\a}^s(0_{\C^3})= \mathcal{Z}=\big\{p \in \mathcal{V}\ \vert\ g(p)=0\big\}=\bigcup\limits_{n \geq 0} \Psi_\a^{-n}(\Omega' \cap \mathcal{Z}).
\]
\end{prop}

\begin{proof}
Let $p \in \mathcal{V}$. 
We know from Proposition \ref{descrdd}
that there exists $C>0$ such that for any 
$j\geq 0$, $|P_\a^{(j)}(p)| \leq C \varphi^j$. 
Then  $p$ belongs to the domain of definition of 
$g$ since in this case,
$|P_\a^{(j)}(p)|^q \varphi^{-j} |\alpha|^{jd} \leq C^q (\varphi^{q-1}|\a|^d)^j$ and 
$\varphi^{q-1}|\a|^d < 1$. It is also clear that if $G_{\Psi_\a}^+(p)>0$, then $p \not\in \mathcal{D}$.

Recall that for any $n\geq 0$,
\begin{equation}\label{groww}
g_n(z)=\varphi z_0 + z_1 + z_2^d \sum\limits_{j=0}^{n} \left(P_\a^{(j)}(z)\right)^q \varphi^{-j} \alpha^{jd}=
\big(P_\a^{(n+1)}(z) + \varphi^{-1}P_\a^{(n)}(z)\big)\varphi^{-n}.
\end{equation}

If $p \in \Omega''=\mathcal{V} \smallsetminus W_{\Psi_\a}^s(0_{\C^3})$, 
then as in the proof of Lemma \ref{lemutile}, we see that 
the sequence $(P_{\alpha}^{(n)}(p) \varphi^{-n})_{n \geq 0}$ converges; moreover,
we get from (\ref{groww}):
\[
\lim\limits_{n\to +\infty} g_n(p)=\lim\limits_{n\to +\infty}  \left(\frac{P_{\a}^{(n+1)}(p)}{P_{\a}^{(n)}(p)}+\varphi^{-1}\right)P_{\a}^{(n)}(p) \varphi^{-n}= \left(\varphi+\varphi^{-1}\right)\lim\limits_{n\to +\infty} P_{\a}^{(n)}(p) \varphi^{-n}.
\]
But we also know that $\lim\limits_{n\to +\infty}  |P_{\a}^{(n)}(p) \varphi^{-n}|>0$, 
so that $g(p)=\displaystyle\lim_{n\to +\infty} g_n(p) \neq 0$ and $p \not\in \mathcal{Z}$. 
This shows that if $p \in \mathcal{Z}$ then $p \in W_{\Psi_\a}^s(0_{\C^3})$; the other implication 
is always true. 

The last point follows from the fact that for $\Omega'\subset \mathcal{V}$, we have 
$\Omega'\cap W_{\Psi_\a}^s(0_{\C^3})=\Omega'\cap \mathcal{Z}$. Moreover, 
$\Omega'$ contains a neighborhood of $0_{\C^3}$ so the orbit of any point 
$p \in W_{\Psi_\a}^s(0_{\C^3})$ will eventually reach $\Omega'$. To conclude, we note that 
by invariance of the stable manifold, we have 
$W_{\Psi_\a}^s(0_{\C^3})=\cup_{n \geq 0} \Psi_\a^{-n}(\Omega'\cap W_{\Psi_\a}^s(0_{\C^3}))$.
\end{proof}

\subsection{Analysis of the dynamics in the case where $\varphi^{(1-q)/d}<\vert\alpha\vert<1$}

Thanks to previous results, we show the following intermediate result 
concerning the dynamics of $\Psi_\a$ in this case.

\begin{prop}\label{proptrichotomy}
Assume $\varphi^{(1-q)/d}<\vert\alpha\vert<1$. For $p\in \mathcal{U}=\{z_2=0\}^c$,
we show the following trichotomy:
\smallskip
\begin{itemize}
\item either $p$ belongs to the stable manifold $W_{\Psi_\a}^s(0_{\C^3})$; in this case, 
its forward iterates converge to $0_{\C^3}$ with exponential speed;
\smallskip
\item or there exist $\varepsilon>0$ and $n_0 \geq 0$ such that for $n \geq n_0$,
\[
|P_\a^{(n)}(p)| \leq ((1-\varepsilon)\varphi)^n;
\]
in this case $p \in \mathcal{Z}$. Furthermore, 
\[
\limsup\limits_{n\to +\infty} |P_\a^{(n)}(p)|^{q-1} |\a|^{nd} > 0;
\]
\item or the orbit of $p$ escapes to infinity very fast: $G_{\Psi_\a}^+(p)>0$.
Moreover the sequence $(|P_\a^{(n)}(p)|)_n$ is increasing after a certain time.
\end{itemize} 
\end{prop}

\begin{proof}
Take $0 < \varepsilon < 1-\varphi^{-1}|\a|^{d/(1-q)}$. Note that 
$\mu:=|\a|^d ((1-\varepsilon)\varphi)^{q-1} > 1$. Let $p=(p_0,p_1,p_2) \not\in W_{\Psi_\a}^s(0_{\C^3})$; 
according to Corollary \ref{corvarstable} its forward orbit is unbounded. 
Suppose that there exists $n_0 \geq 0$ such that for every 
$n \geq n_0$, $|P_\a^{(n)}(p)| \leq ((1-\varepsilon)\varphi)^n$. 
From Lemma \ref{proplessfib}, we know that $p \in \mathcal{Z}$. 
Moreover, if $\limsup\limits_{n\to + \infty} |P_\a^{(n)}(p)|^{q-1} |\a|^{nd} =0$, then with our previous notations, 
$\displaystyle\lim_{n\to+\infty} A(\Psi_\a^n(p)) = \begin{pmatrix}
1 & 1\\
1 & 0
\end{pmatrix}$ and the growth is at least Fibonacci, which is excluded. 
We are then in the second case of Proposition \ref{proptrichotomy}.

Let us handle the remaining case. In particular, we can take $n_0 \geq 0$ 
as big as we want such that $|P_\a^{(n_0)}(p)| > ((1-\varepsilon)\varphi)^{n_0}$. 
Note that since by assumption $|\a|>\varphi^{(1-q)/d}$, and we have 
$(q-1)+d\gamma >0$ (with the notations of 
Lemma~\ref{lemmaspeedexplosion}). In particular, $M=1$ 
satisfies the hypotheses of this lemma. Take $n_0 \geq 0$ sufficiently large such that 
$|P_\a^{(n_0)}(p)|>((1-\varepsilon)\varphi)^{n_0}$ and $\mu^{n_0}|p_2|^d \geq 2+\varphi$.
We can always assume that $|P_\a^{(n_0)}(p)| \geq |P_\a^{(n_0-1)}(p)|$.\footnote{Else
there exists $n_1<n_0$ such that $|P_\a^{(n_1)}(p)| > ((1-\varepsilon)\varphi)^{n_0}$ 
and $|P_\a^{(n_1)}(p)| \geq |P_\a^{(n_1-1)}(p)|$ and we consider $n_1$ instead of $n_0$.}
Since $\left| (P_\a^{(n_0)}(p))^{q-1} (\a^{n_0})^d\right| \geq \mu^{n_0}$,
we deduce
\[
\left|\frac{P_\a^{(n_0+1)}(p)}{P_\a^{(n_0)}(p)}\right|=\left|(P_\a^{(n_0)}(p))^{q-1}(\a^{n_0})^d |p_2|^d+ 
1+\frac{P_\a^{(n_0-1)}(p)}{P_\a^{(n_0)}(p)}\right|\geq\mu^{n_0}|p_2|^d-2\geq\varphi.
\]
This shows that after time at most $n_0$, the sequence $(|P_\a^{(n)}(p)|)_n$ is increasing, 
moreover, there exists $C_0 > 0$ such that for any $n \geq 0$, 
\[
|P_\a^{(n)}(p)| \geq C_0 \varphi^{n}.
\] 
The assumptions of Lemma \ref{lemmaspeedexplosion} are satisfied, 
and we thus get the desired estimate on the speed.
\end{proof}

\begin{rem}
Denote by $\mathcal{S}$ the set of points corresponding to the second case
described in Proposition \ref{proptrichotomy}. A priori, points in $\mathcal{S}$ 
might exhibit a rather complicated dynamics: their forward orbit is not bounded, 
still, it could happen that it does not escape to infinity. We will see that in fact this 
behavior does not occur: $\mathcal{S}=\emptyset$. 
This is related to the properties of the Hénon map $\phi_\a$ to which 
$\Psi_\a$ is semi-conjugate: $\phi_\a$ possesses an attractor at infinity 
which attracts any point 
whose forward orbit is not bounded.
\end{rem}

Let us see how the previous result enables us to conclude the analysis of the 
dynamics of $\Psi_\a$ when $\varphi^{(1-q)/d}<\vert\alpha\vert<1$. Note that in this case, 
$0_{\C^2}$ becomes a saddle point for $\phi_\a$. 
We have seen in \S \ref{subsgeneralrem} that 
$W_{\phi_\a}^s(0_{\C^2})=K_{\phi_\a}^+$ is closed.
Therefore, we recover the fact recalled above in the particular case 
of the saddle fixed point $0_{\C^2}$, and which asserts that  $J_{\phi_\a}^+=\overline{W_{\phi_\a}^s(0_{\C^2})}=W_{\phi_\a}^s(0_{\C^2})$.
Let $p=(p_0,p_1,p_2) \in \C^3$. For $n \geq 0$, 
\begin{equation}\label{eqconjug}
\theta \circ \Psi_\a^n(p)= \big(P_\a^{(n)}(p) (\a^{n}p_2)^l,P_\a^{(n-1)}(p)(\a^{n}p_2)^l,\a^n p_2\big)=\big(\phi_\a^n \circ h (p),\alpha^n p_2\big).
\end{equation}
Recall that $ \mathcal{U}:=\{z_2=0\}^c$ and that  
$\mathcal{S} \subset \mathcal{U}$ denotes the set of points 
whose behavior is described in the second item of Proposition 
\ref{proptrichotomy}. From the estimate 
on the speed we obtained, we know that 
\[
\mathcal{S} \subset \big\{p \in \C^3\ \vert \ G_{\Psi_\a}^+(p)=0\big\}.
\] 
Since $G_{\Psi_\a}^+=G_{\phi_\a} \circ h$, we deduce that 
$h(\mathcal{S}) \subset\big\{\widetilde p\in \C^2\ \vert\ G_{\phi_\a}^+(\widetilde p)=0\big\}=W_{\phi_\a}^s(0_{\C^2})$. 
Assume that $\mathcal{S}$ is non-empty and take $p \in \mathcal{S}$. From (\ref{eqconjug}), and because 
by definition $\mathcal{S} \subset \mathcal{U}$, we see 
that $\displaystyle\lim_{n \to +\infty} |P_{\alpha}^{(n)}(p)| \cdot |\a|^{nl}$ exists and vanishes. 
Since $l=d/(q-1)$, this is in contradiction with the estimate 
$\displaystyle\limsup_{n \to +\infty} |P_{\alpha}^{(n)}(p)|^{q-1} \cdot |\a|^{nd}>0$ given in Proposition 
\ref{proptrichotomy}.
Let us rephrase what we have obtained: 

\begin{prop}
When $\varphi^{(1-q)/d}<\vert\alpha\vert<1$, the automorphism $\Psi_\a$ shares a certain number 
of properties 
with the Hénon automorphism $\phi_\a$. The point $0_{\C^3}$ is a fixed point of $\Psi_\a$ of 
saddle type, and $J_{\Psi_\a}^+:=\partial K_{\Psi_\a}^+=\overline{ W_{\Psi_\a}^s(0_{\C^3})}$. Moreover, 
it follows from the previous discussion that
\begin{equation}
\begin{array}{rcl}
& h & \\
\{z_2=0\} & \to & \{0_{\C^2}\};\\
K_{\Psi_\a}^+=W_{\Psi_\a}^s(0_{\C^3}) & \to & K_{\phi_\a}^+=W_{\phi_\a}^s(0_{\C^2});\\
\big\{p\in \C^3\ \vert\ G_{\Psi_\a}^+(p)>0\big\} & \to & \big\{\widetilde p\in \C^2\ \vert\ G_{\phi_\a}^+(\widetilde p)>0\big\}.
\end{array}
\end{equation}
In this situation, we see that the set $\Omega''$ introduced in the case 
where $0<|\a|<\varphi^{(1-q)/d}$ 
shrinks to the hyperplane $\{z_2=0\}$ which is contracted by $h$; 
in particular, it has empty interior. 
\end{prop}

\subsection{A few words on the case where $|\alpha|=1$}

Note that in this case, the point 
$0_{\C^3}$ is still fixed by $\Psi_\a$ but it is no longer hyperbolic. 
We show the following trichotomy:

\begin{prop}\label{propdyn1}
Let $p\in \mathcal{U}=\{z_2=0\}^c$. 
We have three possibilities:
\begin{itemize}
\smallskip
\item either $p \in K_{\Psi_\a}^+$, that is, its forward orbit is bounded;
\smallskip
\item or $p \in \mathcal{Z}\smallsetminus K_{\Psi_\a}^+$; 
in particular, $|P_\a^{(n)}(p)|=o(\varphi^{n/q})$;
\smallskip
\item or $G_{\Psi_\a}^+(p)>0$; moreover the sequence $(|P_\a^{(n)}(p)|)_n$ is increasing after a certain time.
\end{itemize}
\end{prop}

\begin{proof}
Assume that $p \not\in K_{\Psi_\a}^+$ and $p \in \mathcal{D}$. This implies 
$|P_\a^{(n)}(p)|=o(\varphi^{n/q})$. From Lemma \ref{proplessfib}, and since $q>1$, 
we deduce that $p \in \mathcal{Z}$ and we are 
in the second case. 
 
Let us then assume that $p \not\in \mathcal{D}$ and fix $\epsilon>0$ such that
 $(1-\epsilon)\varphi>1$. 
From Lemma \ref{proplessfib}, we see that for every $n_0 \geq 0$, it is possible to 
find $n \geq n_0$ such that $|P_\a^{(n)}(p)| \geq ((1-\epsilon) \varphi)^n$.
Arguing as in the proof of Proposition \ref{proptrichotomy}, 
we see that the assumptions of 
Lemma \ref{lemmaspeedexplosion} are satisfied after a certain time, and we conclude
that we are in the third case described above.
\end{proof}

\begin{rem}
Note that for any $p \in \C^3$, either its forward orbit escapes to infinity with maximal 
speed (this corresponds to the third case), or $p \in \mathcal{Z}$. We denote by 
$\mathcal{S}'$ the set of points corresponding to the second case described in 
Proposition \ref{propdyn1}. We will show later that in fact $\mathcal{S}'=\emptyset$. 
\end{rem}

When $\alpha^n=1$ for some $n \geq 1$, we see that the dynamics of $\Psi_\a$ is essentially 
given by the one of the Hénon automorphism $\phi=\phi_1$, so we 
assume in the following that $\alpha$ is not a root of unity.

Reasoning as before, (\ref{eqconjug}) tells us that 
$h(K_{\Psi_\a}^+)=K_{\phi_\a}^+=\big\{\widetilde p \in \C^2\ \vert\ G_{\phi_\a}^+(\widetilde p)=0\big\}$, 
but now, $h(W_{\Psi_\a}^s(0_{\C^3}))\subsetneq W_{\phi_\a}^s(0_{\C^2})$. Again, 
$K_{\Psi_\a}^+\neq \big\{p \in \C^3\ \vert\ G_{\Psi_\a}^+(p)=0\big\}$ since there are points in 
$\{z_2=0\}$ escaping to infinity with Fibonacci speed. The point $0_{\C^2}$ is still a saddle 
point of $\phi_\a$, hence $J_{\phi_\a}^+:=\partial K_{\phi_\a}^+=\overline{W_{\phi_\a}^s(0_{\C^2})}$. 
The map $\Psi_\a^{-1}$ is of the same form as $\Psi_\a$, and similarly, we have 
$h(K_{\Psi_\a}^-)=K_{\phi_\a}^-=\big\{\widetilde p \in \C^2\ \vert\ G_{\phi_\a}^-(\widetilde p)=0\big\}$ 
as well as $J_{\phi_\a}^-:=\partial K_{\phi_\a}^-=\overline{W_{\phi_\a}^u(0_{\C^2})}$. 

We define $K_{\Psi_\a}:=K_{\Psi_\a}^+ \cap K_{\Psi_\a}^-$; note that 
$K_{\Psi_\a}\cap \{z_2=0\}=(\Delta_{\varphi}\times \{0\}) \cap (\Delta_{\varphi'}\times \{0\})=\{0_{\C^3}\}$, 
and that $h(K_{\Psi_\a})=K_{\phi_\a}$. We also see that $\restriction{\theta}{\mathcal{U}}$
 maps bijectively $K_{\Psi_\a} \cap \mathcal{U}$ onto $K_{\Phi_\a} \cap \mathcal{U}=K_{\phi_\a} \times \C^*$. 
In particular, $K_{\Psi_{\a}}=\{0_{\C^3}\} \cup \theta^{-1}(K_{\phi_\a} \times \C^*)$. Since 
$\restriction{\theta}{\mathcal{U}}$ is a biholomorphism, we deduce that:
\[
J_{\Psi_\a}:=\partial K_{\Psi_\a}=\partial(\{0_{\C^3}\} \cup\theta^{-1}(K_{\phi_\a} \times \C^*))=\{0_{\C^3}\} \cup
\overline{\theta^{-1} (J_{\phi_\a} \times \C^*)}.
\]

Now, Proposition \ref{propdyn1} implies that for any $p \in \mathcal{U}$, 
either 
$|P_\a^{(n)}(p)|=O(\varphi^{j/q})$ 
(this corresponds to the first and the second cases described in 
this proposition), or $G_{\Psi_\a}^+(p)>0$. With the notations of 
Proposition \ref{propdyn1}, assume that $\mathcal{S}'\neq \emptyset$ and 
let $p \in \mathcal{S}'\subset (K_{\Psi_\a}^+)^c$. 
In particular, $G_{\Psi_\a}^+(p)=0$. 
But $G_{\Psi_\a}^+(p)=G_{\phi_\a}^+(h(p))$, so 
$h(p) \in \big\{\widetilde{p} \in \C^2\ \vert\ G_{\phi_\a}^+(\widetilde{p})=0\big\}=K_{\phi_\a}^+$. 
Then Equation (\ref{eqconjug}) implies that $p \in K_{\Psi_\a}^+$, a contradiction: 
we conclude that $\mathcal{S}'= \emptyset$.

We define a Green function $G_{\Psi_\a}^-$ 
in the same way as we did before, as well as a 
current $T_{\Psi_\a}^-:=\mathrm{dd^c}(G_{\Psi_\a}^-)$.  
We note that 
$\restriction{T_{\Psi_\a}^\pm}{\mathcal{U}}=(\restriction{\theta}{\mathcal{U}})^*(\restriction{T_{\Phi_\a}^\pm}{\mathcal{U}})=(\restriction{h}{\mathcal{U}})^*(\restriction{T_{\phi_\a}^\pm}{\mathcal{U}})$, 
and by construction, the currents $T_{\Psi_\a}^\pm$ satisfy $\Psi_\a^* (T_{\Psi_\a}^\pm)=q^{\pm 1}\cdot T_{\Psi_\a}^\pm$. 
The measure 
\[
\mu_{\Psi_\a}:=T_{\Psi_\a}^+ \wedge 
T_{\Psi_\a}^- \wedge dz_2 \wedge d\overline{z_2}
\]
is invariant by $\Psi_\a$. Moreover, if we denote $\mu_{\Phi_\a}:=\mu_{\phi_\a} \wedge dz_2 \wedge d\overline{z_2}$, then
\begin{equation}\label{equmesur}
\restriction{\mu_{\Psi_\a}}{\mathcal{U}}=(\restriction{h}{\mathcal{U}})^*(\restriction{\mu_{\phi_\a}}{\mathcal{U}}) \wedge dz_2 \wedge d\overline{z_2}=(\restriction{\theta}{\mathcal{U}})^*(\restriction{\mu_{\Phi_\a}}{\mathcal{U}}).
\end{equation}
Since $\mu_{\phi_\a}$ has support in the compact set $J_{\phi_\a}:=\partial K_{\phi_\a}$, 
we deduce from \eqref{equmesur} that $\mu_{\Psi_\a}$ is supported on 
$J_{\Psi_\a}=\{0_{\C^3}\} \cup\overline{\theta^{-1} (J_{\phi_\a} \times \C^*)}$.

For every $p_2 \neq 0$, the set 
$\mathcal{C}_{p_2}:=\C^2 \times\big\{p_2 e^{ix}\ \vert \ x \in \mathbb{R}\big\}$
 is invariant both by $\Psi_\a$ and $\Phi_\a$. 
We know that $(\phi_\a,\mu_{\phi_\a})$ is mixing (in particular, weakly mixing), 
and for any $p_2 \neq 0$, the restriction of $z_2 \mapsto \alpha z_2$ to $\mathcal{C}_{p_2}$ 
is ergodic for $dz_2 \wedge d\overline{z_2}$,
hence $(\restriction{\Phi_\a}{\mathcal{C}_{p_2}},\mu_{\Phi_\a})$ is ergodic 
(see \cite{SinaiBunimovich} for instance). 
We define $\mathcal{J}_{p_2}:=J_{\Psi_\a} \cap \mathcal{C}_{p_2}$; this set is invariant, 
and we know that $\restriction{\mu_{\Psi_\a}}{\mathcal{J}_{p_2}}$ is supported on it. 
By \eqref{equmesur},
we conclude that $(\restriction{\Psi_\a}{\mathcal{J}_{p_2}},\mu_{\Psi_\a})$ 
is ergodic too.  
Yet there is no hope to get mixing properties for $\Psi_\a$
since by projection on the third coordinate, $z_2 \mapsto \alpha z_2$ is a quasiperiodic factor 
of the dynamics. We have thus obtained:

\begin{prop}
For any point $p \in \C^3$, we are in exactly one of the following cases:
\begin{itemize}
\smallskip
\item either the orbit of $p$ is bounded, i.e.\ $p \in K_{\Psi_\a}$;
\smallskip
\item or $p \in \{z_2=0\} \smallsetminus \{0_{\C^3}\}$;
\smallskip
\item or $G_{\Psi_\a}^+(p) >0$ or $G_{\Psi_\a}^-(p)>0$. 
\end{itemize}
\smallskip
The measure $\mu_{\Psi_\a}$ is invariant by $\Psi_\a$ and supported on the Julia set $J_{\Psi_\a}=\{0_{\C^3}\} \cup
\overline{\theta^{-1} (J_{\phi_\a} \times \C^*)}$.
Moreover, when $p_2 \neq 0$ and $\a$ is not a root of unity, 
$(\restriction{\Psi_\a}{\mathcal{J}_{p_2}}, \mu_{\Psi_\a})$ is ergodic.
\end{prop}

\vspace{8mm}

\bibliographystyle{plain}
\bibliography{biblio}
\nocite{}

\end{document}